\documentclass[a4paper, 10pt]{amsart}
\usepackage[top = 1in, bottom = 1in, left=1.2in, right=1.2in]{geometry}
\usepackage{amsthm,amssymb,amsmath,graphicx,mathtools,enumerate,bbm}
\usepackage{aligned-overset}
\usepackage{caption}
\usepackage{subcaption}
\usepackage[shortlabels]{enumitem} 
\usepackage[british]{babel}

\usepackage[mathlines]{lineno}
\usepackage{etoolbox} 

\newcommand*\linenomathpatch[1]{%
	\expandafter\pretocmd\csname #1\endcsname {\linenomath}{}{}%
	\expandafter\pretocmd\csname #1*\endcsname{\linenomath}{}{}%
	\expandafter\apptocmd\csname end#1\endcsname {\endlinenomath}{}{}%
	\expandafter\apptocmd\csname end#1*\endcsname{\endlinenomath}{}{}%
}
\newcommand*\linenomathpatchAMS[1]{%
	\expandafter\pretocmd\csname #1\endcsname {\linenomathAMS}{}{}%
	\expandafter\pretocmd\csname #1*\endcsname{\linenomathAMS}{}{}%
	\expandafter\apptocmd\csname end#1\endcsname {\endlinenomath}{}{}%
	\expandafter\apptocmd\csname end#1*\endcsname{\endlinenomath}{}{}%
}

\expandafter\ifx\linenomath\linenomathWithnumbers
\let\linenomathAMS\linenomathWithnumbers
\patchcmd\linenomathAMS{\advance\postdisplaypenalty\linenopenalty}{}{}{}
\else
\let\linenomathAMS\linenomathNonumbers
\fi

\linenomathpatchAMS{gather}
\linenomathpatchAMS{multline}
\linenomathpatchAMS{align}
\linenomathpatchAMS{alignat}
\linenomathpatchAMS{flalign}
\linenomathpatch{equation}

\usepackage[colorlinks,bookmarks,linkcolor=black,citecolor=black]{hyperref} 

\setlist{leftmargin=*,itemsep=2pt,parsep=1pt,topsep=3pt,partopsep=0pt}  
\setenumerate{leftmargin=*,labelindent=0pt} 

\def\itm#1{\rm ({#1})}
\def\itmit#1{\normalfont (#1)} 
\def\rom{\itmit{\roman{*}}} 
\def\abc{\itmit{\alph{*}}}
\def\arab{\itm{\arabic{*}}} 
\def\itmarab#1{\mbox{\itm{{\it #1\,}\arabic{*}\hspace{.05em}}}}

\def\itmrom#1{\mbox{\itm{{\it #1\,}\roman{*}\hspace{.05em}}}}

\usepackage{tikz}
\usetikzlibrary{positioning}
\usetikzlibrary{backgrounds}
\usetikzlibrary{shapes,decorations,patterns}

\tikzstyle{set} = [minimum size=60,draw=black!80,circle,line width=2,draw opacity=1]
\tikzstyle{white set} = [set, fill=white]
\tikzstyle{blue set}  = [
  postaction={pattern=north east lines, pattern color=blue!60},
  postaction={draw=black!80,circle,line width=2,draw opacity=1},
  fill=blue!30, set]
\tikzstyle{red set} = [set, fill=red!30]
\tikzstyle{vertex} = [draw=black!80,circle,line width=1,fill=black!80]
\tikzstyle{edge} = [on background layer,fill opacity=0.3]
\tikzstyle{blue edge} = [preaction={pattern=north east lines,
  pattern color=blue!60}, blue, edge]
\tikzstyle{red edge} = [red, edge]

\newcommand{\es}{\emptyset}
\newcommand{\eps}{\varepsilon}
\newcommand{\sm}{\setminus}
\renewcommand{\subset}{\subseteq}
\newcommand{\dcup}{\dot\cup}

\theoremstyle{plain}
\newtheorem{theorem}{Theorem}[section]

\newtheorem{corollary}[theorem]{Corollary}
\newtheorem{lemma}[theorem]{Lemma}

\newtheorem{claim}[theorem]{Claim}

\newtheorem{remark}[theorem]{Remark}

\theoremstyle{definition}
\newtheorem{definition}[theorem]{Definition}
\newtheorem{problem}[theorem]{Problem}
\newtheorem{conjecture}[theorem]{Conjecture}

\newcommand{\oldqed}{}
\newcommand{\qedClaim}{\hfill\scalebox{.6}{$\Box$}}
\newenvironment{claimproof}[1][Proof]{
	\renewcommand{\oldqed}{\qedsymbol}
	\renewcommand{\qedsymbol}{\qedClaim}
	\begin{proof}[#1]
	}{
	\end{proof}
	\renewcommand{\qedsymbol}{\oldqed}
} 

\numberwithin{equation}{section}

\def\G{\mathcal{G}}

\def\cP{\mathcal{P}}
\def\T{\mathcal{T}}
\def\V{\mathcal{V}}
\newcommand{\cQ}{\mathcal{Q}}

\newcommand{\Pbri}{P^{\text{Bridge}}}
\renewcommand{\r}{\text{red}}
\renewcommand{\b}{\text{blue}}
\renewcommand{\T}{P}
\newcommand{\exc}{\text{Exc}}
\newcommand{\Pexc}{\T^\text{\exc}}
\newcommand{\Vexc}{V^\text{\exc}}

\newcommand{\Mblow}{\mathcal{M}}

\newcommand{\B}{\mathcal{B}}

\newcommand{\Z}{\mathcal{Z}}
\newcommand{\RdG}{\G}
\newcommand{\Pedge}{P^{\text{Edge}}}

\newcommand{\I}[2]{I_{#1,#2}}
\newcommand{\J}[2]{J_{#1,#2}}

\newcommand{\cI}[2]{\mathcal{I}_{#1,#2}}

\newcommand{\R}{\mathcal{R}}
\newcommand{\cX}{\mathcal{X}}
\newcommand{\cY}{\mathcal{Y}}

\title[Partitioning 2-coloured graphs of min. degree $2n/3 + o(n)$ into monochromatic cycles]{Partitioning a 2-edge-coloured graph of minimum degree $2n/3 + o(n)$ into three monochromatic cycles}
\author[P.~Allen]{Peter Allen}
\author[J.~B\"ottcher]{Julia B\"ottcher}
\author[R.~Lang]{Richard Lang}
\author[J.~Skokan]{Jozef Skokan}
\author[M.~Stein]{Maya Stein}

\address[P.~Allen $\vert$ J.~B\"ottcher $\vert$ J.~Skokan]{London School of Economics, London, WC2A 2AE, UK.}
\email{(p.d.allen$\vert$j.boettcher$\vert$j.skokan)@lse.ac.uk}

\address[R.~Lang]{Institut für Informatik, 69120 Heidelberg, Germany.}
\email{lang@informatik.uni-heidelberg.de}

\address[J.~Skokan]{Department of Mathematics, University of Illinois at Urbana-Champaign,  Urbana, IL 61801, USA.}

\address[M.~Stein]{Departamento de Ingenier\'ia Matem\'atica y Centro de Modelamiento Matemático, Universidad de Chile, Santiago, Chile.}
\email{mstein@dim.uchile.cl}

\thanks{The research leading to these results was partly supported by the DFG grant 450397222 (R.~Lang) and also by Fondecyt Regular Grant 1183080, CONICYT + PIA/Apoyo Basal, C\'odigo AFB170001 (M.~Stein).}

\begin{document}


\begin{abstract}
	Lehel conjectured in the 1970s that every red and blue edge-coloured 
	complete graph can be partitioned into two monochromatic cycles.
	This was confirmed in 2010 by Bessy and Thomass\'e. However, the 
	host graph $G$ does not have to be complete. It it suffices to require 
	that $G$ has minimum degree at least  $3n/4$, where $n$ is the order of
	$G$, as was shown recently by Letzter, confirming a conjecture of 
	Balogh, Bar\'{a}t, Gerbner, Gy\'arf\'as and S\'ark\"ozy. This degree 
	condition is asymptotically tight. 

	Here we continue this line of research, by  proving that for every red 
	and blue  edge-colouring of an $n$-vertex graph of minimum degree at 
	least $2n/3 + o(n)$, there is a partition of the vertex set into three 
	monochromatic cycles. This approximately verifies a conjecture of 
	Pokrovskiy and  is essentially tight.
\end{abstract}
\maketitle

\thispagestyle{empty}

\section{Introduction}
Influentially, Lehel conjectured that every colouring of the edges of the
complete graph~$K_n$ with two colours admits two monochromatic cycles of
distinct colours whose vertex sets partition~$V(K_n)$. Here, an edge, a single
vertex, and the empty set each count as a cycle.  {\L}uczak, R\"odl and
Szemer\'edi~\cite{LRS98} confirmed this conjecture for sufficiently large~$n$
after preliminary work of Gy{\'a}rf{\'a}s~\cite{Gya83}.  Their result was
improved by the first author~\cite{All08}, and finally, Bessy and
Thomass\'e~\cite{BT10} resolved the conjecture for all~$n$.

Since the early 1990s the field of cycle partitioning has evolved in many 
directions. One natural question extensively studied is to what extend 
Lehel's conjecture generalizes to more than two colours. Here, a landmark 
result is due to Erd{\H o}s, Gy{\'a}rf{\'a}s and Pyber~\cite{EGP91}, who 
showed that any $r$-edge-coloured complete graph~$K_n$ admits a 
partition into $ {25r^2 \log r}$ monochromatic cycles and conjectured 
that $r$ cycles suffice. The upper bound was later improved to 
$100 r \log r$ by  Gy{\'a}rf{\'a}s, Ruszink{\'o}, S{\'a}rk{\"o}zy and 
Szemer{\'e}di~\cite{GRSS06}  for large~$n$. Pokrovskiy \cite{Pok14} 
disproved the conjecture of Erd{\H o}s, Gy{\'a}rf{\'a}s and Pyber by 
providing an $r$-edge-colouring of~$K_n$ that needs~$r+1$ cycles for 
every $r\ge 3$ and infinitely many~$n$. So far, it is far from clear 
what is the best possible number of cycles, even for small~$r$, showing 
that this problem is difficult.

Other variants of the problem that received much attention in the past 
decade include the appearance of subgraphs other than cycles, the use 
of colourings with an unbounded number of colours (but local 
restrictions on the colourings), and analogues for hypergraphs, to 
highlight just a few. We refer the reader to the recent survey of 
Gy\'arf\'as~\cite{Gya16} and the references therein for more detail.
Another important direction, which is the variant we are most 
interested in here, concerns the replacement the complete graph~$K_n$ 
by a non-complete host graph. This has been investigated for almost 
complete graphs~\cite{Let20}, complete bipartite graphs~\cite{Hax97,LSS17,LS17}, 
graphs with fixed independence number~\cite{Sar11}, infinite 
graphs~\cite{ESSS17,Rad78} and random graphs~\cite{KMN+18,LL21} among 
others.  We are interested in the problem for graphs of large minimum degree.

The study of minimum degree conditions for spanning substructures has a 
long tradition in extremal graph theory (see for example the survey~\cite{KO14} 
for an overview), and so it is natural to seek a variant of Lehel's conjecture 
for host graphs of large minimum degree. Balogh, Bar\'{a}t, Gerbner, 
Gy\'arf\'as and S\'ark\"ozy~\cite{BBG+14} conjectured the following: 
For any $2$-edge-colouring of the edges of any $n$-vertex graph~$G$ of 
minimum degree~$3n/4$, there are two distinctly coloured monochromatic 
cycles which together partition the vertices of~$G$. Examples show that 
this would be tight. In support of their conjecture, they proved an 
approximate version in which~$G$ has minimum degree~$3n/4+o(n)$ and the 
cycles are allowed to miss~$o(n)$ vertices. DeBiasio and Nelsen~\cite{DN17} 
showed that under this (stronger) degree condition a complete partition is possible.
Finally, Letzter~\cite{Let19} resolved the full conjecture for all sufficiently large~$n$.

Motivated by these advances, Pokrovskiy~\cite{Pok16} conjectured that analogous results are true for graphs of lower minimum degree.
In particular, he conjectured that any 2-edge-coloured graph of minimum degree~$2n/3$ can be partitioned into 3 monochromatic cycles.
The construction in Figure~\ref{fig:sharpness} shows that this bound is essentially tight.
Our main result confirms Pokrovskiy's conjecture approximately.

\begin{figure}
  \begin{center}
    \begin{tikzpicture}[scale=0.8]
      \tikzstyle{set} = [minimum size=60,draw=black!80,circle,line width=2,draw opacity=1]
      \tikzstyle{white set} = [set, fill=white]
      \tikzstyle{blue set}  = [
	postaction={pattern=north east lines, pattern color=blue!60},
	postaction={draw=black!80,circle,line width=2,draw opacity=1},
	fill=blue!30, set]
      \tikzstyle{red set} = [set, fill=red!30]
      \tikzstyle{vertex} = [draw=black!80,circle,line width=1,fill=black!80]
      \tikzstyle{edge} = [on background layer,fill opacity=0.3]
      \tikzstyle{blue edge} = [preaction={pattern=north east lines,
	  pattern color=blue!60}, blue, edge]
      \tikzstyle{red edge} = [red, edge]
      
      \node[white set] (a) at (0,0) {};
      \node[] (aa) at (-3,0) {$m$};
      \node[set,fill opacity=0.3, fill=red] (a) at (0,0) {};
      \node[preaction={pattern=north east lines,
	  pattern color=blue!60}, edge, set,fill opacity=0.3, fill=blue] (a) at (0,0) {};
 
      \node[blue set] (b) at (4,0) {};
      \node[] (bb) at (7,0) {$m+2$};
      \node[red set] (c) at (0,4) {};
      \node[] (cc) at (-3,4) {$m+2$};
      \node[vertex, label={$a$}] (v) at (3.5,4.5) {};
      \node[vertex, label={$b$}] (w) at (4.5,3.5) {};
      \node (K) at (-1.5,5) {$K$};
      \node (K') at (5.5,-1) {$K'$};
      \node (K'') at (-1.5,-1) {$K''$};
      
      \begin{scope}[on background layer]
	\fill[blue edge] (v) to (a.320) -- (a.120) -- (v) -- cycle;
	\fill[blue edge] (v) to (c.80) -- (c.300) -- (v) -- cycle;
        
	\fill[red edge] (w) to (a.320) -- (a.120) -- (w) -- cycle;
	\fill[red edge] (w) to (b.150) -- (b.10) -- (w) -- cycle;
        
	\fill[red edge] (a.north) to (b.north) -- (b.south) -- (a.south) -- (a.north) -- cycle;
	\fill[blue edge] (a.west) to (c.west) -- (c.east) -- (a.east) -- (a.west) -- cycle;
      \end{scope}
    \end{tikzpicture}
  \end{center}
  \caption{A graph on $3m+6$ vertices, consisting of three vertex disjoint
    cliques $K$, $K'$ and $K''$ and two additional vertices~$a$ and~$b$. Here,
    $K$ and $K'$ have $m+2$ vertices each, all edges of~$K$ are coloured red and
    all edges of~$K'$ are coloured blue. The third clique~$K''$ is of size~$m$
    and its edges can be coloured either red or blue. Further, all edges
    between~$K$ and~$K''$ are present and coloured blue and all edges
    between~$K'$ and~$K''$ are present and coloured red. The vertex~$a$ has blue
    edges to $K\cup K''$ and the vertex~$b$ has red edges to $K'\cup K''$.}
  \label{fig:sharpness}
\end{figure}
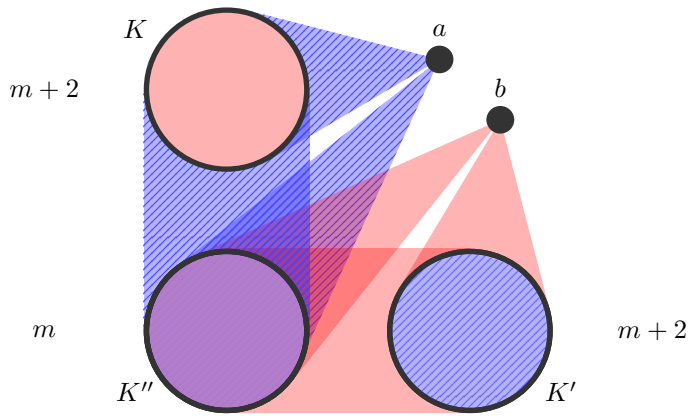

\begin{theorem}\label{thm:main}
  For any~$\beta > 0$, there is~$n_0$ such that for any colouring with
  two colours of the edges of  any  graph $G$  on $n\geq n_0$ vertices 
  and minimum degree at least $(2/3+\beta)n$, there is a partition of 
  $V(G)$ into three monochromatic cycles.
\end{theorem}

The proof of this theorem is partly based on tried and tested techniques such as
graph regularity, the use of matchings, and absorption, but also draws from
several new ideas which could prove useful for future research in this area.  For
an outline of our approach, see Section~\ref{sec:methods}.

\paragraph{Organization.}

The remainder of this paper is organized as follows.  In the next section we
outline our proof method.  In Section~\ref{sec:preliminaries} we then introduce
some notation and concepts related to the regularity method.  In
Section~\ref{sec:big-proof} we present the proof of Theorem~\ref{thm:main} 
alongside the main lemmas needed for this proof.
Sections~\ref{sec:component-lemma},~\ref{sec:distribute-exceptional-vertices},~\ref{sec:balancing}
and~\ref{sec:extremal-colourings} are dedicated to the Lemmas concerning finding
the components, distributing the exceptional vertices and solving the cases of
extremal colourings, respectively. Finally, we offer some open problems and
conjectures in Section~\ref{sec:remarks}.

\section{Proof outline}\label{sec:methods}
In this section we shall briefly outline our proof, which was inspired by a method
developed by Garbe, Lang, Lo, Mycroft and Sanhueza-Matamala~\cite{GLL+22} to solve cycle
partitioning problems in hypergraphs.

Suppose that $G$ is an $n$-vertex graph of minimum degree at least $2n/3 + o(n)$
and that each of its edges is coloured red or blue. We start by applying the
regularity lemma (Lemma~\ref{lem:regularity}) to obtain a regular partition and
the corresponding reduced graph~$\RdG$.  Roughly speaking, our strategy is to
distinguish the case that~$\RdG$ is close to an extremal configuration, which
can be solved with arguments tailored to this particular structure (see
Lemma~\ref{lem:extremal-colouring}), or we can find three monochromatic
components~$C_1$,~$C_2$,~$C_3$ in~$\RdG$ whose union~$C$ has the
following property (see Lemma~\ref{lem:components}). Every stable set~$S$ in~$C$
has linearly many more neighbours in~$\RdG$ than~$|S|$ (see
Definition~\ref{def:robustMatching}).  From the LP duality of maximum fractional
matchings and minimum vertex covers we obtain that~$C$ admits a
\emph{perfect $2$-matching} in a robust way, where a perfect $2$-matching is a
spanning union of vertex disjoint edges and cycles (see
Theorem~\ref{thm:tutte-2-matching} and the discussion preceding~it).

As is canonical, we then connect the clusters of the components in the 
perfect $2$-matching by short monochromatic paths (see
Lemma~\ref{lem:connecting-path}), which we can further require to contain
certain \emph{exceptional vertices} (see
Lemma~\ref{lem:distribute-exceptional-vertices}). Creating these paths will
create some imbalances between the clusters of the matching edges, which 
we correct by solving a weighted matching problem in an auxiliary graph 
and using the robustness of the perfect $2$-matching (see 
Lemma~\ref{lem:balancing}). To finish the proof, we apply the blow-up 
lemma (Lemma~\ref{lem:blowup}) to find monochromatic spanning paths in 
each of the components of the perfect $2$-matching. Together with the 
short paths constructed earlier this yields the desired cycle partition.

We remark that our approach has the advantage that it avoids a number 
of technicalities from earlier approaches, and allows us to isolate the 
main steps (finding the components, distributing exceptional vertices 
and solving extremal configurations) cleanly in separate lemmas.

\section{Preliminaries}\label{sec:preliminaries}
In this section we introduce some notation and tools, which we will need for the
proof of Theorem~\ref{thm:main}.

\subsection{Notation}
We let $[n]$ denote the set $\{1,2,\dots,n\}$.
Let $G=(V,E)$ be a graph.
The \emph{order} of $G$ is $|V(G)|$ and the \emph{size} of $G$ is $|E(G)|$.
We denote the \emph{neighbourhood} of a vertex $v$ by $N_G(v)$ and write $N_G(v,W) = N_G(v) \cap W$ for a set of vertices $W \subset V(G)$.
We denote the \emph{degree} of $v$ by $\deg_G(v) = |N_G(v)|$ and set $\deg_G(v,W) = |N_G(v) \cap W|$.
For a set of vertices $S \subset G$ we write $N_G(S) = \bigcup_{s \in S} N(s)$.
If it is clear from the context, we often drop the index $G$. We write $G[W]$
for the subgraph of~$G$ induced by the vertices in~$W$ and for a set~$U$
disjoint from~$W$ we write $G[W,U]$ for the bipartite subgraph of~$G$ containing
all edges in $E(G)$ with one end in~$U$ and one in~$W$.
For another graph~$H$ we denote by $G\cup H$ the graph on vertex set $V(G) \cup
	V(H)$ and edge set $E(G) \cup E(H)$. We write $H\subset G$ if~$H$ is a subgraph of~$G$.
A \emph{stable set} in~$G$ is a set of vertices $S\subset V(G)$ such that there
are no edges with both endpoints in~$S$.
For vertices~$v$ and~$w$, a $v,w$-\emph{path (walk)} is a path (walk) that
starts in $v$ and ends in $w$. Similarly, for vertex sets~$W$ and~$W'$, a
$W,W'$-path (walk) is a path (walk) that starts in~$W$ and ends in~$W'$.

Suppose that $G$ is a graph whose edges are coloured with red and blue.
We denote by $G_\r$ ($G_\b$) the subgraph on $V(G)$ that contains the red (blue) edges.
For a vertex~$v$ of~$G$ we also write $N_\r(v)$, $N_\b(v)$, $\deg_\r(v)$,
$\deg_\b(v)$ for $N_{G_\r}(v)$, $N_{G_\b}(v)$, $\deg_{G_\r}(v)$,
$\deg_{G_\b}(v)$ for the neighbourhood/degree of $v$ in the respective colour.
We will say that $w \in N_\r(v)$ is a \emph{red neighbour} of $v$ and likewise for colour blue.
A \emph{red (blue) component} of $G$ is a component of~$G_\r$~($G_\b$).
In particular a vertex $v$ with $\deg_\r(v) = 0$ is in a red component of order 1.
For convenience of notation, we also count the empty graph as a component.

In our coloured graphs we do occasionally allow parallel edges (more precisely,
this is the case for our so-called reduced graphs) -- however, there will always
be at most one red edge and at most one blue edge between any pair of vertices.
For such a multi-graph~$G$ we denote by $\delta(G)$ the minimum degree of the
underlying simple graph, and by $\deg_G(u)$ the degree of~$u$ in the underlying
simple graph.

\subsection{Perfect 2-matchings}

We make use of the concept of a $2$-matching in a graph $G$.
\begin{definition}[Perfect $2$-matching]
	A \emph{perfect $2$-matching} of a graph $G$ is a spanning subgraph of $G$ which is the disjoint union of cycles and single edges.
\end{definition}

The next theorem is a convenient analogue of Tutte's classical characterization
of perfect matchings for $2$-matchings. It follows easily from the fact that
maximum fractional matching and minimum vertex cover are LP dual problems both
of which have optimal solutions which are half-integral, and the LP-duality
theorem. For a proof of this result see, e.g., \cite[Corollary~30.1a]{Sch03}. 

\begin{theorem}[Tutte's Theorem for $2$-matchings]\label{thm:tutte-2-matching}
	A graph $G$ has a perfect $2$-matching if and only if every stable set $S \subset V(G)$ satisfies $|N(S)| \geq |S|$.
\end{theorem}

In order to apply this in our proof, we will need a more robust version of the condition $|N(S)| \geq |S|$. This motivates the following definition.

\begin{definition}[$\gamma$-robust Tutte]
	\label{def:robustMatching}
	We say that a graph~$G$ on~$n$ vertices is \emph{$\gamma$-robust Tutte} if
	for every non-empty stable subset $S$ of $V(G)$ we have $|N(S)| \geq |S| + \gamma n$.
\end{definition}

\subsection{Hamiltonicity}

A \emph{Hamilton cycle} (\emph{Hamilton path}) in a graph~$G$ is a cycle (path)
that covers all the vertices of~$G$.
Chv\'atal's theorem provides a degree sequence criterium for the existence of Hamilton cycles.

\begin{theorem}[Chv\'atal~\cite{Chvatal}]\label{thm:chvatal}
	Let $G$ be a graph with vertex degree sequence $d_1 \leq\dots\leq d_n$.
	If for every $1 \leq i<n/2$ we have $d_i\geq i+1$ or $d_{n-i}\geq n-i$, then
	$G$ has a Hamilton cycle.
\end{theorem}

As a corollary we get the following version for balanced bipartite graphs, which we shall use in our proof.

\begin{corollary}[Bipartite Chv\'atal]\label{cor:bipartite-chvatal}
	Let~$H$ be a bipartite graph with bipartition classes~$X$ and~$Y$, each of size $n$.
	Let $X$ have vertex degree sequence $x_1 \leq\dots\leq x_n$ and $Y$ have vertex
	degree sequence $y_1 \leq\dots\leq y_n$. If for every $i \in [n-1]$ we have
	$x_i\geq i+1$ or $y_{n-i}\geq n-i+1$, then $H$ is Hamiltonian.
\end{corollary}
\begin{proof}
	Obtain $H'$ from $H$ by adding edges between vertices of $Y$ until $H'[Y]$ is complete.
	So every vertex in $Y$ gains $n-1$ new neighbours.
	Let $d_1, \ldots, d_{2n}$ be the degree sequence of $H'$.
	By assumption we have either $x_{n-1}\ge n$ or $y_1\ge 2$. In either case, $y_i \geq 2$ for all $i\in[n]$.
	Hence, for each $i\in[n]$ we have $d_i = x_i$ and $d_{2n-i+1} = y_{n-i+1}+n-1$.
	It follows that $d_i \geq i+1$ or $d_{2n-i} \geq 2n-i$ for $i \in [n-1]$.
	By Theorem~\ref{thm:chvatal}, $H'$ has a Hamilton cycle $C$, which,
	as $|X| =|Y|$, has no edges within $Y$.
	Thus $C$ is a  Hamilton cycle of~$H$.
\end{proof}

This also implies the following result on Hamilton paths.

\begin{corollary}\label{cor:bip-path}
	Let~$H$ be a bipartite graph with bipartition classes $X_1$, $X_2$,
	each of size~$n$. If
	for each $i\in\{1,2\}$, and for each $u\in X_i$ we have
	$\deg_H(u,X_{3-i})\ge\frac12n+1$, then for any  $w\in X_1$, $w'\in X_2$
	there is a Hamilton $w,w'$-path in~$H$.
\end{corollary}
\begin{proof}
	Obtain the
	graph~$H'$ from $H$ by adding a new vertex~$v$ to~$X_1$, a new vertex~$v'$ to~$X_2$
	and the edges $wv',v'v,vw'$. It is easy to verify that the conditions of
	Corollary~\ref{cor:bipartite-chvatal} are satisfied, which gives a
	Hamilton cycle in~$H'$. As this cycle clearly uses the edges
	$wv',v'v,vw'$ it contains Hamilton $w,w'$-path in~$H$.
\end{proof}

We  use Corollary~\ref{cor:bip-path} to obtain Hamilton paths in the following somewhat specialized
setting, which appears in the analysis of our extremal cases.

\begin{lemma}\label{lem:bip-path}
	Let $0<\mu\le 1$, $n\ge100\mu^{-1}$, and let~$H$ be an $n$-vertex graph on $V(H)=A\dcup B$ with $A'\subset A$, $|B|\le|A|\le|B|+\mu
		n$ and $|A'|\le\mu n$, such that
	\begin{enumerate}[label=\rom]
		\item for each $v\in A$ we have
		      $\deg_H(v,B)\ge(\frac14+3\mu)n$,
		\item for each $v\in B$ we have
		      $\deg_H(v,A)\ge(\frac14+3\mu)n$, and
		\item for each $v\in A\setminus A'$ we have
		      $\deg_H(v,A)\ge 3\mu n$.
	\end{enumerate}
	Let $u,u'\in V(H)\setminus A'$ be distinct vertices
	such that
	if both $u,u'\in A\setminus A'$ then $|A|\ge|B|+1$.
	Then there exists a Hamilton $u,u'$-path in~$H$.
\end{lemma}
\begin{proof}
	We start paths $P$ in~$u$ and $P'$ in~$u'$.
	If $u,u'\in B$, we extend $P$ by adding an edge from~$u$ to an
	arbitrary neighbour in $A\setminus A'$.
	Similarly, if $u,u'\in A\setminus A'$, we extend $P'$ by adding an edge from~$u'$ to an
	arbitrary neighbour in $B$.
	Finally, if $|A\cap\{u, u'\}|=|B\cap\{u, u'\}|=1$, then
	we assume without loss of generality that $u\in A$ and $u'\in B$ and do not
	perform any extensions.

	Observe that after this first step, we have
	$|A\setminus V(P\cup P')|\ge|B\setminus V(P\cup P')|$, because of the assumption in the last sentence of Lemma~\ref{lem:bip-path}.
	We then greedily extend the path $P$
	inside the set $A\setminus A'$ until we have $|A\setminus V(P\cup P')|=|B\setminus V(P\cup P')|$, which is possible because $|A'|\le\mu n$ and for
	$v\in A\setminus A'$ we have $\deg_H(v,A)\ge 3\mu n$.
	Since $|A|\le|B|+\mu n$, we now have $V(P)\le \mu n+1$.

	Remove from $A$ and $B$ all vertices from the paths $P$ and $P'$ except their endvertices $w$ and $w'$, thus obtaining sets $A^*$ and $B^*$. Note that $n^*:=|A^*|=|B^*|\le(\frac12+\mu)n$. We apply Corollary~\ref{cor:bip-path} to the bipartite graph $H^*$ between $A^*$ and $B^*$ (it is not difficult to see that the degree conditions are satisfied), in order to find a Hamiltonian path $P''$ in $H^*$ connecting $w$ and $w'$. We finally connect $P', P''$ and $P'$ to form the desired Hamilton $u,u'$-path in~$H$.
\end{proof}

\subsection{Regularity}\label{sec:regularity}
Given a graph $G$ and disjoint vertex sets $U, W \subset V(G)$ we denote 
the number of edges between $U$ and $W$ by $e(U,W)$ and the \emph{density} 
of $(U,W)$ by $d(U,W) = e(V,W)/(|U||W|)$. The pair $(U,W)$ is called
\emph{$\eps$-regular} in~$G$ if, for all subsets $X \subset U$, 
$Y\subset W$ with $|X| \geq \eps |U|$ and $|Y| \geq \eps |W|$, we have
$|d(U,W)-d(X,Y)| \leq \eps$. We say that $(U,W)$ is \emph{$(\eps,d)$-regular} 
if it is $\eps$-regular and has density at least~$d$.
We say that a vertex $u \in U$ has \emph{typical degree} in a regular 
pair $(U,W)$ if $\deg(u,W) \geq (d(U,W)-\eps)|W|$. It follows directly 
from the definition of $\eps$-regularity that if $(U,W)$ is $\eps$-regular, 
then
\begin{equation}\label{equ:typical-degree-in-reg-pair}
  \text{all but at most $\eps |U|$ vertices in~$U$ have typical degree in $(U,W)$.}
\end{equation}
Another straightforward consequence of the definition of $\eps$-regularity
concerns large subpairs of regular pairs: If $(U,W)$ is an $(\eps,d)$-regular
pair and $U'\subset U$, $W'\subset W$ are of size $|U'|\ge\frac12|U|$,
$|W'|\ge\frac12|W|$, then
\begin{equation}\label{eq:slicing}
	\text{$(U',W')$ is a $(2\eps,d/2)$-regular pair.}
\end{equation}

Szemer\'edi's regularity lemma allows us to partition the vertex set of a graph
into clusters of vertices  such that most pairs of clusters are regular~\cite{Sze76}.
We will use the regularity lemma in its degree form and with 2 colours
(see~\cite{KS96}). For a coloured graph~$G$, a colour~$\chi$ and two disjoint
vertex sets~$U$, $W$ of~$G$ we write $(U,W)_\chi$ for the bipartite subgraph
of $G[U,W]$ containing all edges of colour~$\chi$.

\begin{lemma}[Regularity lemma]\label{lem:regularity}
	For each $\eps>0$ and $m_0\in\mathbb N$ there is $M = M(\eps,m_0)$ such that  for any $2$-edge colouring of any $n$-vertex graph~$G$, and for each $d>0$ there is a partition
	$V_0,V_1,\ldots,V_m$ of $V(G)$ and a subgraph $G'$ of $G$ with vertex set
	$V(G') \sm V_0$ such that:
	\begin{enumerate}[label=\abc]
		\item $m_0 \leq m \leq M$,
		      \item\label{itm:reg-size} $|V_0| \leq \eps n$ and $|V_1| = \ldots = |V_m| \leq \lceil\eps n \rceil$,
		\item $\deg_{G'}(v) \geq \deg_G(v)-(2d+\eps)n$ for each $v \in V(G) \sm V_0$,
		\item $G'[V_i] $ has no edges for $i \in [m]$,
		\item \label{itm:reg-lemma-eps-regular} for all pairs $(V_i,V_j)$ and each
		      colour~$\chi$ the pair $(V_i,V_j)_\chi$
		      is $\eps$-regular in~$G'$ and has density either 0 or at least $d$.
	\end{enumerate}
\end{lemma}

Let $G$ be a red and blue edge coloured graph with a partition $V_0, \ldots,V_m$
obtained from Lemma~\ref{lem:regularity} with parameters $\eps$, $m_0$ and $d$.
We then also call the family $\V=\{V_1,\dots,V_m\}$ an \emph{$(\eps,d)$-regular
	partition} of~$G$ with \emph{exceptional set} $V_0$.
We define the (coloured) $(\eps,d)$-\emph{reduced multi-graph}~$\RdG$ corresponding to this
partition to be the multi-graph graph with vertex set $V(\RdG) = [m]$
in which two vertices
$i$ and $j$ are connected by a red (or blue) edge, if $(V_i,V_j)$ is an
$\eps$-regular pair of density at least $d$ in red (blue, respectively). Note
that it is  possible that some $i$ and $j$ are connected by both a red and a
blue edge.

It is easy to show that $\RdG$ inherits the minimum degree of $G$.
More precisely, if $G$ has minimum degree $cn$, it follows that
\begin{equation}\label{equ:reduced-graph-min-deg}
	\delta(\RdG)\ge(c-2d-\eps)m.
\end{equation}

Further, for a graph~$G$, vertex disjoint sets $V_1,\ldots,V_\ell$ of $V(G)$ and a
graph~$\mathcal H$ on vertex set $[\ell]$, we say that~$G$ is
\emph{$(\eps,d)$-regular on~$\mathcal H$} if $(V_i,V_j)$ is $(\eps,d)$-regular
whenever $ij\in E(\mathcal{H})$.

The next lemma allows us to connect clusters by short paths provided 
their counterparts in the reduced graph lie in the same component.
Though the lemma as stated here does not explicitly appear in previous work,
this technique is standard, and the lemma is straightforward to prove.
Given a family of sets $\V= \{V_1,\ldots,V_t\}$, we say that a set $W$ (or a
graph $D$ with $V(D)=W$) is $\delta$-\emph{sparse} in $\V$ if $|W \cap V_i| \leq
	\delta |V_i|$ for each $i \in [t]$.

\begin{lemma}[Connection lemma]\label{lem:connecting-path}
	Let~$\eps,d>0$ and let $m, \ell, n\in\mathbb N$ such that $10\eps\le d\le\frac1{10}$ and $\frac{10m\cdot\ell}{d}\le n$.
	Let $G$ be a graph on~$n$ vertices, let~$V_0\subseteq V(G)$ with $|V_0|\le\eps n$, let $\V = \{V_1,\ldots,V_m\}$ be an equipartition of $V(G)\setminus V_0$,
	and let~$\RdG$ be a graph on~$[m]$ so that~$G$ is $(\eps,d)$-regular
	on~$\RdG$.
	Let $W=(w_1,\ldots,w_\ell)$ be a walk of order~$\ell$ in $\RdG$.
	Let $X \subset V_{w_1}, Y \subset V_{w_\ell}$ be subsets of size $d|V_{1}|/2$.
	Let $S \subset V(G)$ be $(d/4)$-sparse in $\V$.
	Then there is an $X,Y$-path $P \subset G-S$ of order~$\ell$.
\end{lemma}

Finally, the blow-up lemma allows us to find almost spanning cycles or paths in
certain regular partitions.
A pair $(U,W)$ is \emph{$(\eps,d)$-super-regular} if it is $(\eps,d)$-regular
and, in addition, for all vertices $u\in U, w\in W$, we have $\deg(u,W)\ge d|U|$
and $\deg(w,U)\ge d|W|$.
The following version of the blow-up lemma can be found
in~\cite[Lemma~7.1]{blowup}.
Note that we only require the special case $\Delta(\RdG')\leq 2$ here.
We remark that one could replace our use of Lemma~\ref{lem:blowup} with a use of the
original blow-up lemma of Koml\'os, S\'ark\"ozy and Szemer\'edi~\cite{KSS97} and
some additional technical work.
\begin{lemma}[Blow-up lemma]\label{lem:blowup}
	For all $\Delta\ge 2$, $\alpha,\zeta, d>0$, $\kappa>1$ there exist
	$\eps,\rho>0$ such that for all $M$ the following holds for all sufficiently
	large $n$.

	Let $\RdG$ be a graph on~$[m]$ with $m\le M$ and let
	$\RdG'$ be a spanning subgraph of $\RdG$ with $\Delta(\RdG')\leq 2$. Let~$H$, $G$ be graphs whose vertex sets are partitioned into sets $X_1,\dots,X_m$
	and $V_1,\dots,V_m$ respectively, such that $|X_i|=|V_i|\ge \frac n{\kappa m}$ and $|X_i|\le\kappa|X_j|$ for each $i,j\in[m]$. Let
	$\widetilde{X}_i\subset X_i$ satisfy $|\widetilde{X}_i|\ge\alpha|X_i|$ for
	each $i\in[m]$, and let $I_x\subset V_i$ for each $i\in[m]$ and $x\in X_i$.
	Suppose that
	\begin{enumerate}[label=\itmarab{B}]
		\item\label{B:X} $\Delta(H)\leq \Delta$, all edges of $H$ lie on pairs
		$X_iX_j$ with $ij\in\RdG$, and for each $i\in[m]$, all edges of $H$ incident with $\widetilde{X}_i\cup N(\widetilde{X}_i)$ lie on pairs $X_jX_{j'}$ with $jj'\in
			\RdG'$,
		\item\label{B:V} $G$ is $(\eps,d)$-regular on~$\RdG$, and for each $ij\in
			E(\RdG')$ the pair $(V_i,V_j)$ is $(\eps,d)$-super-regular, and
		\item\label{B:I} for each $i\in[m]$ and $x\in X_i$ we have
		$|I_x|\ge\zeta|X_i|$ and there are at most $\rho|X_i|$ vertices $x\in X_i$
		such that $I_x\neq V_i$.
	\end{enumerate}
	Then there is an embedding $\psi\colon V(H)\to V(G)$ with $\psi(x)\in
		I_x$ for each $x\in H$.
\end{lemma}

\begin{figure}
  \begin{center}
  \begin{tikzpicture}[scale=0.7]
    \begin{scope}[transform shape]      
      \node[white set, minimum width=48] (a) at (0,0) {};
      \node[white set, minimum width=69] (b) at (4,0) {};
      \node[white set, minimum width=65] (c) at (0,4) {};
      \node[white set, minimum width=54] (d) at (4,4) {};
      \begin{scope}[on background layer]
	\fill[blue edge] (a.south) -- (b.south) -- (b.north) -- (a.north)  -- (a.south) -- cycle;
	\fill[blue edge] (c.north) -- (d.north) -- (d.south) -- (c.south)  -- (c.north) -- cycle;
	\fill[red edge] (a.west) -- (a.east) -- (c.east) -- (c.west) -- cycle;
	\fill[red edge] (a.west) -- (a.east) -- (d.east) -- (d.west) -- cycle;
	\fill[red edge] (b.west) -- (b.east) -- (c.east) -- (c.west) -- cycle;
	\fill[red edge] (b.west) -- (b.east) -- (d.east) -- (d.west) -- cycle;
	\fill[red edge] (b.east) -- (d.east) -- (d.south) -- (c.south)  -- (c.west) -- (b.west) -- (b.east) --  cycle;
      \end{scope}
    \end{scope}
    \begin{scope}[xshift=9cm,transform shape]
      \node[red set] (a) at (0,0) {};
      \node[blue set] (b) at (4,0) {};
      \node[blue set] (c) at (0,4) {};
      \node[red set] (d) at (4,4) {};
      \begin{scope}[on background layer]
	\fill[red edge] (a.south) -- (b.south) -- (b.north) -- (a.north) -- (a.south) --  cycle;
	\fill[red edge] (c.south) -- (d.south) -- (d.north) -- (c.north) -- (c.south) --  cycle;
	\fill[blue edge] (a.west) -- (c.west) -- (c.east) -- (a.east) -- (a.west) --  cycle;
	\fill[blue edge] (b.west) -- (d.west) -- (d.east) -- (b.east) -- (b.west) --  cycle;
      \end{scope}
    \end{scope}
  \end{tikzpicture}
  \end{center}
  \caption{The extremal colourings: The configuration given in
    Definition~\ref{def:extremal-colouring}\ref{conf:bipartite} is depicted on
    the left and the one in
    Definition~\ref{def:extremal-colouring}\ref{conf:four-cycle} on the right.}
  \label{fig:extremal-colourings}
\end{figure}
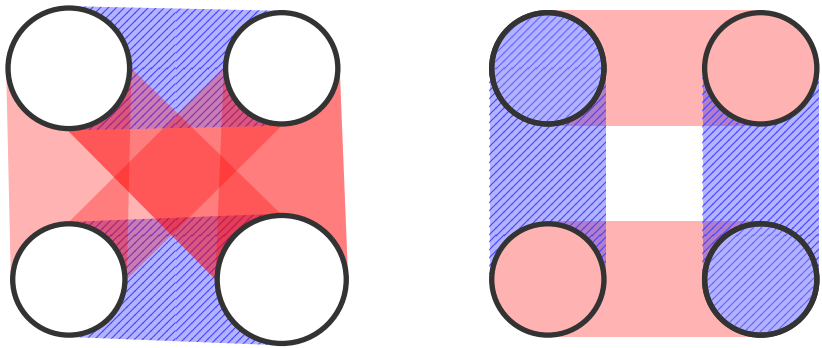


\section{Main lemmas and proof of the main theorem}\label{sec:big-proof}

In this section we prove Theorem~\ref{thm:main}. We first sketch the proof and
give the lemmas we need, and then give a formal proof.

For proving Theorem~\ref{thm:main}, we apply the regularity lemma to $G$ to
obtain a regular partition $\V=\{V_1,\dots,V_m\}$ with exceptional set~$V_0$ and
a corresponding coloured reduced multi-graph~$\RdG$ on vertex set~$[m]$.
Note that~$\RdG$ approximately inherits the (relative) minimum degree conditions as~$G$.
Our task is now to analyse the structure of monochromatic components in $\RdG$.

We begin by introducing the following extremal colourings (see also Figure~\ref{fig:extremal-colourings}), which require special treatment.
Recall that in a coloured multi-graph there is at most one red edge and at most one blue edge between any pair of vertices.

\begin{definition}[Extremal colourings]\label{def:extremal-colouring}
	Let $G$ be an $m$-vertex (multi-)graph whose edges are coloured with red and
	blue. We say that the colouring of $G$ is \emph{$\gamma$-extremal} if one of
	the following holds (modulo swapping colours).
	\begin{enumerate}[label=\abc]
		\item \label{conf:bipartite} $G$ has a spanning bipartite red component with
		      bipartition classes $X_1,X_2$ and such that $\big||X_1|-|X_2|\big|\leq
			      \gamma m$. There are (exactly) two blue components $B_1,B_2$, both
		      bipartite, with $V(B_1) = X_1$ and $V(B_2) = X_2$.
		\item \label{conf:four-cycle}
		      The subgraph $G$ has two red and two blue monochromatic components $R_1,R_2,B_1,B_2$ such that $V(G) = \bigcup_{i,j \in [2]} V(R_i) \cap V(B_j)$ and $|V(R_i) \cap V(B_j)| \leq (1/4 + \gamma)m$ for $1\leq i,j \leq 2$.
			      {Moreover, all but at most $\gamma m^2$ edges of $G[V(R_i) \cap V(B_j)]$ are blue if $i=j$ and red otherwise.}
	\end{enumerate}
\end{definition}

Now we present our main structural lemma, whose proof we defer to Section~\ref{sec:component-lemma}.
It states that if $\RdG$ is not close to an extremal colouring, then we can choose three
monochromatic components covering all vertices of $\RdG$ which have various good
properties that allow us to complete the cycle partition.

\begin{lemma}[Finding components]\label{lem:components}
	Let $1/m\ll\gamma$.
	Let $\RdG$ be a multi-graph on $m$ vertices  with $\delta(\RdG) \geq (2/3 +
		8\gamma)m$, whose edges are coloured in two colours.
	Suppose that the colouring of $\RdG$ is not $(4\gamma)$-extremal.
	Then there are distinct monochromatic components $C_1,C_2,C_3 \subset\RdG$ such that $C :=
		\bigcup C_i$ spans $\RdG$ and all of the following hold.
	\begin{enumerate}[label=\rom]
		\item 
		      \label{itm:robust-matching}
		      $C$ is $\gamma$-robust Tutte.
		\item 
		      \label{itm:overlap} One of the following holds:
		      \begin{enumerate}[label=\abc]
			      \item \label{itm:true-overlap}
			            $|\bigcup_{i\neq j} V(C_i) \cap V(C_j)| \geq (1/3 + \gamma)m$, or
			      \item \label{itm:spanning-and-empty} $C_1$ is spanning, $C_1$ or $C_2$
			            contains an odd cycle, and $C_3 = \es$.
		      \end{enumerate}
		\item 
		      \label{itm:odd-monochromatic-cycle}
		      If $C_1,C_2,C_3$ are each bipartite, then $C_3=\es$.
		\item 
		      \label{itm:connectivity} $C$ is connected.
	\end{enumerate}
\end{lemma}

Our next task is to make a case distinction, depending on the structure of
$\RdG$ and of $G$. For this we need the following definition.

\begin{definition}[Bridges]\label{def:bridges}
	Suppose $G$ is a graph whose edges are coloured in red and blue and which has
	an $(\eps,d)$-regular partition
	$\V=\{V_1,\ldots,V_m\}$ with corresponding coloured $(\eps,d)$-reduced multi-graph $\RdG$.
	Let $H_1, H_2 \subset \RdG$ be monochromatic connected subgraphs of $\RdG$ of
	colour $\chi$ and
	let $G_\chi$ be the subgraph of~$G$ formed by all edges in $E(G)$ of colour $\chi$.
	We say that $(H_1,H_2)$ \emph{admits bridges} if there are two distinct
	vertices $u,u'$ in~$G$ and clusters $V_i$, $V_j$, $V_{i'}$, $V_{j'}$ in~$\V$
	with $i,i'\in V(H_1)$ and $j,j'\in V(H_2)$
	such that
	\begin{align*}
		\deg_{G_\chi}(u,V_i)     & \ge d|V_i|\,,                   & \qquad \deg_{G_\chi}(u,V_j)     & \ge d|V_j|,       \\
		\deg_{G_\chi}(u',V_{i'}) & \ge d|V_{i'}|\,\qquad\text{and} & \qquad \deg_{G_\chi}(u',V_{j'}) & \ge d|V_{j'}| \,.
	\end{align*}
	If the colour~$\chi$ is blue, then we also talk of \emph{blue bridges}, if $\chi$ is red we talk of \emph{red bridges}.
\end{definition}

Note that for any monochromatic connected $H  \subset \RdG$ with at least one edge, the pair
$(H,H)$ admits bridges.

Continuing with our proof strategy, we separate three cases. First, $\RdG$ is
not close to either extremal colouring. Second, $\RdG$ is close to an extremal
colouring but there is a pair of red (blue) components of $\RdG$ which admit
red (blue) bridges. Third, $\RdG$ is close to an extremal colouring and no such
bridges exist. In the last case, the structure of $G$ is so close to the
respective extremal graph that we can find the desired three-cycle partition without the regularity lemma, as the following lemma states. We prove this lemma in
Section~\ref{sec:extremal-colourings}.

\begin{lemma}[Extremal cases]\label{lem:extremal-colouring}
	Let ${1}/{n}  \ll \eps,1/m \ll d \ll \gamma \ll \beta.$
	Suppose $G$ is a graph on $n$ vertices with minimum degree $\delta(G) \geq (2/3 +
		\beta)n$ whose edges are coloured with red and blue.
	Suppose $G$ has a coloured $(\eps,d)$-reduced multi-graph $\RdG$ on $m$ vertices.
	Let the colouring of $\RdG$ be $\gamma$-extremal and such that one of the following is true:
	\begin{itemize}
		\item $\RdG$ has a colouring as in Definition~\ref{def:extremal-colouring}\ref{conf:bipartite} with monochromatic components $R_1,B_1,B_2$ and $(B_1,B_2)$ does not admit bridges.
		\item $\RdG$ has a colouring as in Definition~\ref{def:extremal-colouring}\ref{conf:four-cycle} with monochromatic components $R_1,R_2,B_1,B_2$ and neither $(R_1,R_2)$ nor $(B_1,B_2)$ admits bridges.
	\end{itemize}
	Then $G$ has a partition into three monochromatic cycles.
\end{lemma}

When $\RdG$ is not close to an extremal colouring, or when it is but some pair
of components admits bridges, we work with the regular partition to find the
desired three-cycle partition. We now describe the proof for the case that
$\RdG$ is not close to either extremal colouring, and explain briefly at the end
the modification required for the case that $\RdG$ is close to an extremal
colouring but some pair of components admit bridges. We set $H_i:=C_i$ for
$i=1,2,3$, where the $C_i$ are the components provided by
Lemma~\ref{lem:components}.

The first step is to fix a perfect $2$-matching $\Mblow$
of~$\RdG$.  We
then move a few vertices from each of the clusters
$V_1,\dots,V_m$ to the exceptional set in order to make the regular pairs of $G$ corresponding to
$\Mblow$ super-regular. We denote the resulting exceptional set by $\Vexc$.

Our next step is to find three pairwise disjoint cycles in $G$ which form the
basis of our three-cycle partition. In each $H_i$ we will find a cycle $D_i$, such that all the $D_i$ together cover the vertices~$\Vexc$.
In addition, the cycles $D_i$ will be `easy to extend to the pairs in $H_i$', which we make precise in
next definition.

\begin{definition}[Support path embeddings]\label{def:spe}
	Let $G$ be a graph whose edges are coloured in red and blue and which has an $(\eps,d)$-regular partition $\V=\{V_1,\dots,V_t\}$ with corresponding coloured $(\eps,d)$-reduced multi-graph $\RdG$.
	Let~$H \subset \RdG_\chi$ be a subgraph of colour $\chi$.
	We say that $D \subset G_\chi$ \emph{{supports path embeddings in $H$}}, if for each edge $ij \in E(H)$ there is an edge $f_{ij}=\{u,v\} \in D$ with $u\in V_i$ and $v\in V_j$ such that~$u$ and~$v$ have typical degree in the regular pair $(V_i,V_j)_\chi$.
	We also say that the edge~$f_{ij}$ \emph{witnesses} that~$D$ supports path embeddings in~$H$ for~$ij$.
\end{definition}

Our next lemma constructs the desired cycles covering the exceptional vertices.
We prove this lemma in Section~\ref{sec:distribute-exceptional-vertices}.

\begin{lemma}[Covering exceptional vertices]\label{lem:distribute-exceptional-vertices}
	Let ${1}/{n}  \ll \eps,\frac{1}{m} \ll \delta \ll d  \ll {\gamma} \ll \beta$.
	Suppose $G$ is a graph on $n$ vertices of minimum degree $\deg(G) \geq (2/3 +
		\beta)n$ whose edges are coloured with two colours.
	Let $G$ have an $(\eps,d)$-regular partition $\V = \{V_1,\ldots,V_m\}$
	with exceptional set~$V_0$ and corresponding coloured $(\eps,d)$-reduced multi-graph $\RdG$.
	Let $W \subset V(G)$ be $\delta$-sparse in $\V$.
	Let $H_1,H_2,H_3$ be subgraphs of $\RdG$ with the following
	properties.
	\begin{enumerate}[label=\rom]
		\item\label{itm:distr-exc-i} One of the following holds:
		\begin{enumerate}[label=\abc]
			\item \label{itm:distr-exc-intersection} $|\bigcup_{i\neq j} V(H_i) \cap V(H_j)| \geq {(1/3 +
					      \gamma)m}$ or
			\item \label{itm:distr-exc-spanning-and-empty} $H_1$ is spanning, $H_1$ or
			      $H_2$ contains an odd cycle, and $H_3 = \es$.
		\end{enumerate}
		\item If $H_1,H_2,H_3$ are each bipartite, then $H_3=\es$.
		\item \label{itm:distr-exc-H1H2H3} $H_1$ and $H_3$ are {monochromatic components} and $H_2$ is the union of two {monochromatic components} which admit bridges.
		\item \label{itm:distr-exc-different-colours} If $V(H_i)\cap V(H_j)\neq\emptyset$, then $H_i$ and $H_j$ are of different colours.
	\end{enumerate}
	Then there are disjoint monochromatic cycles $D_1$, $D_2$, $D_3$ such that $D_i$ supports paths embeddings in $H_i$ for each $i \in [3]$.
	Moreover, for $D = D_1 \cup D_2 \cup D_3$, the set~$V(D)$ is $\sqrt{\delta}$-sparse in $\V$, we have $W \subset V(D)$, and $|V(G) \sm V(D)|$ is even.
\end{lemma}

To complete the proof, we aim to extend the cycles $D_1,D_2,D_3$ to
monochromatic cycles partitioning $V(G)$. To do this, we will replace in each
$D_i$ some of the edges $f_e$ (as in Definition~\ref{def:spe}) with longer
paths $P_e$ of the correct colour which are contained in the regular pair
corresponding to $e\in E(\RdG)$ such that, overall, we cover all vertices
of~$G$. Any such path uses the same number~$t'_e$ of
vertices in each of the two clusters, which is the reason why we require in the
conclusion of Lemma~\ref{lem:distribute-exceptional-vertices} that $|V(G)\sm
	V(D)|$ is even.

But how do we choose the numbers~$t'_e$?
For explaining this, we remark first that our strategy is to embed all the paths
$P_e$ vertex-disjointly using only one application of the blow-up lemma
(Lemma~\ref{lem:blowup}) with $\RdG'=\Mblow$.
We need $t'_e$ to be large for all
edges $e\in\Mblow$ (for satisfying condition~\ref{B:X} of
Lemma~\ref{lem:blowup}).
More precisely, we let $z=\lfloor\tfrac{n}{4m}\rfloor$ and require for all
edges~$e$ that
$t'_e\ge\omega(e)z$  where $\omega(f)=2$ if $e$ is on an edge of $\Mblow$, $\omega(f)=1$ if $e$ is on a cycle of $\Mblow$, and $\omega(e)=0$ otherwise.  Changing variables by setting
\begin{equation}\label{eq:change-vars}
	t_e=t'_e-\omega(e)z=t'_e-\omega(e)\lfloor\tfrac{n}{4m}\rfloor\,,
\end{equation}
this is equivalent to requiring~$t_e\ge0$.
Moreover, as we need to
cover all vertices in each cluster $V_i\in\RdG$,  we require
\begin{equation}\label{eq:cover-all}
	\big|V_i\setminus V(D)\big| = \sum_{e\in E(\RdG)\colon i\in e} t'_e
	= \sum_{e\in E(\RdG)\colon i\in e} (t_{e}+\omega(e)z)
	= 2z + \sum_{e\in E(\RdG)\colon i\in e} t_{e}
	\,.
\end{equation}
The following lemma guarantees that such non-negative values $t_e$ exist. When
we apply this lemma, we shall set $t_i=\big|V_i\setminus V(D)\big|-2z$, hence
all these~$t_i$ will be very similar in size.

\begin{lemma}[Balancing lemma]\label{lem:balancing}
	Given $m$ and $0<\gamma\le\frac12$, if $t\ge 5m/\gamma$ the following holds. Let
	$H$ be a connected graph on~$m$ vertices that is $\gamma$-robust Tutte. Let
	$(t_i)_{i\in V(H)}$ be integers such that $\sum_it_i$ is even and such that
	$t_i=\big(1\pm \tfrac15\gamma\big)t$ for each $i\in V(H)$. Then there
	exist non-negative integers $(t_e)_{e\in E(H)}$ such that for each $i\in V(H)$
	we have
	$\sum_{e\in E(H)\colon i\in e}t_e=t_i$.
\end{lemma}

We prove Lemma~\ref{lem:balancing} in Section~\ref{sec:balancing}. This
concludes our proof sketch in the case that $\RdG$ is not close to an extremal
colouring.

The final case to be considered is that $\RdG$ is close to an extremal colouring
but two components admit bridges. In this case we let $H_2$ be the union of
these two components, and let $H_1$, $H_3$ be the other two components (it is
possible that $H_3=\es$). It turns out that the above proof sketch also works in
this setting.

We are now ready to give the details of the proof of Theorem~\ref{thm:main}.

\begin{proof}[Proof of Theorem~\ref{thm:main}]
	We will start by fixing some constants. Given $\beta>0$ we may
	assume that $\beta\le 1/3$. We  choose $\gamma,d,\delta>0$, and $m_0$ such that
	\begin{equation}\label{eq:const1}
		\frac{1}{m_0}\ll \delta \ll d\ll \gamma\ll\beta
	\end{equation}
	are small enough so that we can apply
	Lemma~\ref{lem:components} with input $\frac{1}{m_0}\ll \gamma$,
	Lemma~\ref{lem:extremal-colouring}
	with input $\frac{1}{m_0} \ll d\ll 4\gamma\ll\beta$
	and
	Lemma~\ref{lem:distribute-exceptional-vertices}
	with input $\frac{1}{m_0}\ll\delta \ll d \ll\beta$.
	The blow-up lemma (Lemma~\ref{lem:blowup}) with input
	\[\Delta=2, \quad \alpha=\tfrac15, \quad \zeta=d/2, \quad d/2 \quad\text{and}\quad \kappa=2\]
	provides us with constants $\eps'>0$ and $\rho>0$.
	We choose $\eps$ such that $0<2\eps\le\eps'$,
	and
	\begin{equation}\label{eq:const2}
		\eps\ll\delta.
	\end{equation}
	Note that our choices guarantee we can use
	Lemma~\ref{lem:extremal-colouring}
	with input $\eps,\frac{1}{m_0} \ll d\ll 4\gamma\ll\beta$
	and
	Lemma~\ref{lem:distribute-exceptional-vertices}
	with input $\eps,\frac{1}{m_0}\ll\delta \ll d \ll\beta$.
	The regularity lemma (Lemma~\ref{lem:regularity}) provides us with a
	constant~$M$
	for input~$\eps$ and~$m_0$.
	Finally, we choose~$n_0$ such that
	\begin{equation}\label{eq:const3}
		n_0^{-1}\ll M^{-1}\le m_0^{-1}
	\end{equation}
	is sufficiently small for Lemma~\ref{lem:blowup},
	Lemma~\ref{lem:extremal-colouring}, and
	Lemma~\ref{lem:distribute-exceptional-vertices} (with inputs as set above, but
	with~$M$ instead of~$m_0$).
	Assume that $n\ge n_0$.
	We further assume that our choice of constants in~\eqref{eq:const1},
	\eqref{eq:const2}, and \eqref{eq:const3} is such that these constants are
	chosen sufficiently small in terms of each other for various estimates we will
	use in the proof, for example for the application of Lemma~\ref{lem:balancing}
	below.

	Given an $n$-vertex graph $G$ with minimum degree
	$\delta(G)\ge\big(\tfrac23+\beta\big)n$, whose edges are coloured with red and
	blue, let $\V=\{V_1,\dots,V_m\}$ be a regular partition of $G$ with
	exceptional set~$V_0$, as provided by Lemma~\ref{lem:regularity} for input
	$\eps,m_0,d$. We have $m_0\le m\le M$. Let $\RdG$ be the corresponding
	coloured $(\eps,d)$-reduced multi-graph of $G$.
	By~\eqref{equ:reduced-graph-min-deg} and
	by the choice of $\gamma$ we have $\delta(\RdG)\ge\big(\tfrac23+8\gamma\big)m$.
	Hence we can apply Lemma~\ref{lem:components} with input~$m$ and~$\gamma$ to the
	coloured multi-graph~$\RdG$, which tells us that $\RdG$ is $(4\gamma)$-extremal or there are monochromatic
	components $C_1,C_2,C_3\subset\RdG$ such that $C :=\bigcup C_i$ spans~$\RdG$ and
	\ref{itm:robust-matching}--\ref{itm:connectivity} of
	Lemma~\ref{lem:components} hold.
	We distinguish three cases and define subgraphs $H_1,H_2,H_3 \subset \RdG$ as follows:
	\begin{enumerate}[label=\abc]
		\item\label{case:gen} $\RdG$ is not $(4\gamma)$-extremal. In this case we set $H_i:=C_i$ for each $i=1,2,3$.
		\item\label{case:nearex} $\RdG$ is $(4\gamma)$-extremal and some pair of components admit bridges.
		\begin{enumerate}[label=\arab]
			\item If $\RdG$ is close to the colouring of Definition~\ref{def:extremal-colouring}\ref{conf:bipartite}, with one red component $R_1$ and two blue components $B_1,B_2$ which admit bridges, we set $H_1:=R_1$, $H_2:=B_1\cup B_2$ and $H_3=\es$.
			\item If $\RdG$ is close to the colouring of Definition~\ref{def:extremal-colouring}\ref{conf:bipartite}, with one blue component $B_1$ and two red components $R_1,R_2$ which admit bridges, we set $H_1:=B_1$, $H_2:=R_1\cup R_2$ and $H_3=\es$.
			\item If $\RdG$ is close to the colouring of
			      Definition~\ref{def:extremal-colouring}\ref{conf:four-cycle}, with two red
			      components $R_1,R_2$ and two blue components $B_1,B_2$
			      such that $(B_1,B_2)$ admits bridges, we set $H_1:=R_1$, $H_2:=B_1\cup B_2$ and $H_3=R_2$.
			\item If $\RdG$ is close to the colouring of
			      Definition~\ref{def:extremal-colouring}\ref{conf:four-cycle}, with two
			      red components $R_1,R_2$ and two blue components $B_1,B_2$
			      such that $(R_1,R_2)$ admits bridges,
			      we set $H_1:=B_1$, $H_2:=R_1\cup R_2$ and $H_3=B_2$.
		\end{enumerate}
		\item\label{case:done} None of the above cases apply.
	\end{enumerate}
	In case~\ref{case:done}, by Lemma~\ref{lem:extremal-colouring} there exist three monochromatic cycles which partition the vertices of~$G$, as desired. We may therefore assume one of the other two cases holds. Let~$H$ be the uncoloured spanning subgraph of~$\RdG$ whose edges are all pairs lying in at least one of~$H_1$,~$H_2$ and~$H_3$. We claim that (in either of the two cases) all of the following hold.
	\begin{enumerate}[label=\itmrom{\normalfont H:}]
		\item\label{H:robust-matching} $H$ is $\gamma$-robust Tutte.
		\item\label{H:overlap} One of the following holds
		\begin{enumerate}[label=\abc]
			\item\label{H:true-overlap} $|\bigcup_{i\neq j} V(H_i) \cap V(H_j)| \geq (1/3 + \gamma)m$ or
			\item\label{H:spanning-and-empty} $H_1$ is spanning, $H_1$ or $H_2$ contains an odd cycle, and $H_3 = \es$.
		\end{enumerate}
		\item\label{H:odd-monochromatic-cycle} If $H_1,H_2,H_3$ are each bipartite, then $H_3=\es$.
		\item\label{H:comps} $H_1$ and $H_3$ are monochromatic components of $\RdG$, while $H_2$ is the union of two monochromatic components of $\RdG$ of the same colour which admit bridges.
		\item \label{H:intersect-distinct-colour}  If $V(H_i)\cap V(H_j)\neq\emptyset$, then $H_i$ and $H_j$ are of different colours.
		      \item\label{H:connectivity} $H$ is connected.
	\end{enumerate}

	These statements hold in case~\ref{case:gen} by Lemma~\ref{lem:components},
	with each $H_i$ being a distinct monochromatic component of $\RdG$ (and so $H_2=C_2\cup
		C_2$, and as observed after Definition~\ref{def:bridges}, $(C_2,C_2)$ admits bridges trivially). In
	case~\ref{case:nearex}, $H_1$ and $H_3$ are distinct monochromatic components of $\RdG$
	and $H_2$ is by construction the union of two monochromatic components of
	$\RdG$ of the same colour which admit bridges, hence we get~\ref{H:comps}.
	Further, there are no other other monochromatic components in~$\RdG$ by
	Definition~\ref{def:extremal-colouring}, hence we get~\ref{H:connectivity} and
	$H=\RdG$ has minimum
	degree $\delta(H)\ge 2m/3$. It follows that no stable set has more than $m/3$
	vertices and every stable set has a neighbourhood of size at least $2m/3$,
	giving~\ref{H:robust-matching}. It is easy to verify
	that also~\ref{H:overlap}\ref{H:true-overlap}
	and~\ref{H:odd-monochromatic-cycle} hold.
	From this point on, we do not need to distinguish cases~\ref{case:gen}
	and~\ref{case:nearex}; the above list of properties is all we need.

	By~\ref{H:robust-matching} and Theorem~\ref{thm:tutte-2-matching}, there is an (uncoloured)
	perfect $2$-matching  $\Mblow$ in $H$. Let $\Vexc$ consist of $V_0$ together with, for each $i\in V(\RdG)$, all
	vertices $v\in V_i$ such that $v$ does not have typical degree in any pair
	$(V_i,V_j)$ with $ij\in E(\Mblow)$ in any colour such that $\RdG$ has an edge
	$ij$ of this colour. By~\eqref{equ:typical-degree-in-reg-pair}, there are at
	most $4\eps|V_i|$ such vertices in each $V_i$, so that $\Vexc$ is
	$\delta$-sparse in $\V=\{V_1,\dots,V_m\}$.

	We apply Lemma~\ref{lem:distribute-exceptional-vertices} to~$G$ with the
	partition~$\V$ with reduced multi-graph~$\RdG$, and with $W=\Vexc$. The
	conditions of the lemma are satisfied
	by~\ref{H:overlap},~\ref{H:odd-monochromatic-cycle} and~\ref{H:comps}. We
	obtain pairwise vertex disjoint monochromatic cycles $D_1,D_2,D_3$ in $G$, such
	that for each $i$ the colour of $D_i$ is the same as the colour of $H_i$, and
	$D_i$ supports path embeddings in $H_i$. Furthermore, for $D:=D_1\cup D_2\cup
		D_3$ the set $V(D)$ is $\sqrt{\delta}$-sparse in $\V$, we have $\Vexc\subset V(D)$,
	and
	\begin{equation}\label{eq:main-even}
		\big|V(G)\setminus V(D)\big| \quad \text{is even}\,.
	\end{equation}
	For each $e\in H$, choose some $p(e)\in\{1,2,3\}$ such that $e\in
		E(H_{p(e)})$, and let~$f_{e}$ be the edge of~$D_{p(e)}$ witnessing that
	$D_{p(e)}$ supports path embeddings for~$e$.

	It remains to extend the cycles~$D_1$, $D_2$, $D_3$ to cycles covering all
	of~$G$, for which we shall use the blow-up lemma. This lemma requires a reduced
	graph without multi-edges, and hence we fix a (coloured) subgraph~$\RdG^*$ of
	$\RdG$ by deleting for each multi-edge in~$\RdG$ one of the two parallel edges,
	with the constraint that if $e\in H$ then the edge of the same colour as $f_e$
	is kept. Observe that this does not create problems with connectivity:
	Since~$D_i$ supports path-embeddings in~$H_i$ and we shall use this to
	extend~$D_i$ by merely replacing some of its edges by longer paths, the
	resulting graph will still be a cycle. Hence it only remains to check that we
	can do these extensions so that all vertices of~$G$ get covered. For this we shall use
	the blow-up lemma with reduced graph $\RdG^*$ and with $\RdG'=\Mblow$.

	Let us next check that the vertices not covered by~$D_1$, $D_2$, and~$D_3$
	indeed form a regular partition which is super-regular on the pairs
	corresponding to $\Mblow$, so that we can apply the blow-up lemma. For each
	$i\in [m]$ let $V_i^*=V_i\setminus V(D)$ and let $G^*$ be a graph on vertex set
	$V(G)\setminus V(D)$ that contains for each $ij\in\RdG^*$ of some colour~$\chi$
	the edges of~$G$ in the regular pair $(V_i,V_j)$ of colour~$\chi$ that have one
	endpoint in~$V^*_i$ and the other in~$V^*_j$. Consider the partition
	$\V^*=\{V^*_1,\dots,V^*_m\}$. Observe that since $\Vexc\subset V(D)$, for
	each $ij\in\Mblow$ and for any vertex in $v\in V^*_i$ we have
	$$\deg_{G^*}(v,V^*_j)\ge(d-\eps)|V_j|-\sqrt{\delta}|V_j|>d|V^*_j|/2.$$
	Moreover, by~\eqref{eq:slicing}
	for each edge $ij\in E(\RdG^*)$
	the pair $\big(V^*_i,V^*_j\big)$ in~$G^*$ is $(2\eps,d/2)$-regular.
	In other words, $G^*$ with the partition $\V^*$ is
	$(2\eps,d/2)$-regular on $\RdG^*$ and $\big(V^*_i,V^*_j\big)$ is
	$(2\eps,d/2)$-super-regular for each $ij\in\Mblow$.
	Hence~$G^*$, with the partition~$\V^*$ and the graph~$\RdG^*$ with subgraph
	$\RdG'=\Mblow$, satisfies Lemma~\ref{lem:blowup}\ref{B:V}.

	As explained before, for edges $e\in H$ we shall replace the edge $f_e$ (which
	is part of some~$D_i$) by a path in the same colour of length~$t'_e$. Let us
	now choose the values~$t'_e$ with the help of Lemma~\ref{lem:balancing}. To
	this end, set $z:=\lfloor\tfrac{n}{4m}\rfloor$ and for each $i\in[m]$, set
	$t_i:=\big|V^*_i\big|-2z$.
	Observe that $\sum_{i\in[m]}t_i=|V(G^*)|-2mz$, which is even
	by~\eqref{eq:main-even}. Moreover, since
	$(1-\eps)n/m\le|V_i|\le n/m$ holds by
	Lemma~\ref{lem:regularity}\ref{itm:reg-size}, and since $V(D)$ is
	$\sqrt{\delta}$-sparse in $\V$, we have $(1-\eps)n/m-\sqrt{\delta} n/m\le
		|V^*_i|\le n/m$ and hence $t_i=(1\pm 3\sqrt{\delta})n/(2m)=(1\pm
		\frac15\gamma)n/(2m)$ for each $i\in V(H)$. Since~$H$ is $\gamma$-robust Tutte
	by~\ref{H:robust-matching} and connected by~\ref{H:connectivity}, we can
	apply Lemma~\ref{lem:balancing} to obtain non-negative integers $(t_e)_{e\in
				E(H)}$ such that $\sum_{e\in E(H):i\in e}t_{e}=t_i$
	holds for each $i$. Setting $t'_e:=t_e+\omega(e) z$,  where $\omega(f)=2$ if $e$ is on an edge of $\Mblow$, $\omega(f)=1$ if $e$ is on a cycle of $\Mblow$, and $\omega(e)=0$ otherwise, we obtain
	\begin{equation}\label{eq:main-te}
		\sum_{e\in E(H)\colon i\in e}t'_{e}=\big|V^*_i\big|
		\quad
		\text{and}
		\quad
		t'_{\hat e}\ge z\ge\frac{n}{5m}
	\end{equation}
	for each $i\in[m]$ and each $\hat e\in\Mblow$.

	We now construct a graph~$F$ containing the paths that
	we shall embed with the help of the blow-up lemma, and define a partition
	$V(F)=X_1\cup\dots\cup X_m$ as follows. For each $e=ij\in H$ with $t'_e>0$, we
	add a path $P_e$ with $2t'_e$ vertices to $F$, whose vertices we put
	alternately to $X_i$ and $X_j$. Thus all edges of $F$ lie on pairs $(X_i,X_j)$
	such that $ij\in H$. For each $i\in V(H)$, we let $\widetilde{X}_i\subset X_i$
	be the vertices lying on paths between $X_i$ and some $X_j$ such that $ij\in\Mblow$.
	By~\eqref{eq:main-te} we have $|\widetilde{X}_i|\ge\frac{n}{5m}\ge|V^*_i|/5$
	and $|X_i|=|V^*_i|$ for each $i\in[m]$. Hence~$F$ and the partition
	$X_1\cup\dots\cup X_m$ satisfy Lemma~\ref{lem:blowup}\ref{B:X}.

	Finally, we set up image restrictions~$I_x$ for $x\in V(F)$ as follows, to
	guarantee that when we embed the paths in~$F$ they connect appropriately to the
	corresponding~$D_i$. If $x\in X_i$ has degree~$2$ in $F$, we set $I_x=V^*_i$.
	Otherwise, $x$ is an end-vertex of a path whose vertices are in $X_i\cup X_j$
	and by definition, the edge $e=ij$ lies in~$H_{p(e)}$. In this case let~$z$ be
	the end-vertex of $f_e\in E(D_{p(e)})$ that is in $X_j$ and let~$I_x$ be
	the $G^*$-neighbours~$z'$ of~$z$ in~$X^*_i$. Observe that by definition
	of~$G^*$ each such edge~$zz'$ has the same colour as $f_e$.
	Since~$D_i$ supports path embeddings in~$H_i$, the vertex~$z$ is typical in
	$(V_i,V_j)$ and hence has at least $(d-\eps)|V_i|-\sqrt{\delta}|V_i|\ge\frac12 d|V^*_i|$
	neighbours in~$V^*_i$.
	Further, since for each edge $e$ of $\RdG^*$ we choose a unique $f_e\in E(D)$,
	for each $i\in[m]$ there are at most $m-1\le \rho|X_i|$ vertices $x\in X_i$ with $I_x\neq
		V^*_i$. Hence the $I_x$ satisfy Lemma~\ref{lem:blowup}\ref{B:I}.

	Thus, applying Lemma~\ref{lem:blowup}, we find an embedding~$\phi$ of~$F$
	into~$G^*$ such that for each $x\in V(F)$ we have $x\in I_x$. We  combine
	the embedding of~$F$ with the cycles $D_1,D_2,D_3$, by replacing for each $e\in
		H$ such that $t'_e>0$ the edge $f_e=uv$ on the cycle $D_{p(e)}$ with a path
	as follows. Assume without loss of generality that $u\in V_i$ and $v\in V_j$
	and that the first vertex~$w$ of $\phi(P_e)$ is in~$V_j$ and the last
	vertex~$w'$ of $\phi(P_e)$ is in~$V_i$. Then we replace $f_e$ by the path
	$u,\phi(P_e),v$ to obtain a longer monochromatic cycle. This is possible since
	$\phi(P_e)$ has the same colour $\chi$ as~$f_e$ by the definition of~$G^*$, and since, by the definition of the sets
	$I_x$, also  $uw$ and $w'v$  have colour $\chi$.
	We thus obtain three monochromatic cycles
	$D'_1,D'_2,D'_3$, whose vertices  partition $V(G)$.
\end{proof}


\section{Finding components: The proof of Lemma~\ref{lem:components}}\label{sec:component-lemma}
This section is dedicated to the proof of Lemma~\ref{lem:components}.
Given $\gamma >0$, let $G :=\RdG$ be a red and blue edge coloured graph on $n :=m$ vertices and of minimum degree at least $(2/3 + 8\gamma) n$.
Suppose that the colouring of $G$ is not $(4\gamma)$-extremal.
Our goal is to find monochromatic components $C_1,C_2,C_3 \subset G$  whose union spans $G$ and satisfies \ref{itm:robust-matching}--\ref{itm:connectivity} of Lemma~\ref{lem:components}.

We start by showing that $G$ is spanned by two monochromatic components.
\begin{claim}\label{cla:cover-by-C1-C2}
	There are monochromatic components $C_1, C_2$ that together span $G$.
\end{claim}
\begin{claimproof}
	Let $\mathcal C$ be the set of monochromatic components of $G$.
	Let $v_1, \ldots,v_q$ be a maximal number of vertices such that for distinct $i,j\in [q]$, vertices $v_i$ and~$v_j$ are both in distinct red and in distinct blue components. Then \[
		\frac{2}{3}nq  \leq \sum_{i\in[q]} \deg(v_i) =  \sum_{i\in[q]} d_\text{red}(v_i) + \sum_{i\in[q]} d_\text{blue}(v_i) \leq 2(n - q),
	\] and therefore, $q\le 2$.
	Now consider the bipartite graph $H$ whose partition classes are the red and blue components respectively and which has an edge $CC'$ whenever and $V(C)\cap V(C')\neq\emptyset$.
	Observe that the size of any matching in $H$ is at most $q\le 2$.
	So by K\H onig's theorem, the edges of~$H$ are covered by two vertices. These vertices are monochromatic components spanning $G$.
\end{claimproof}

We fix monochromatic components $C_1$ and $C_2$ that together span the vertices of $G$ and distinguish three cases:
\begin{description}[labelwidth=*]
	\item[Case 1] One of $C_1$ and $C_2$ is spanning,
	\item[Case 2] $C_1$ and $C_2$ have distinct colours,
	\item[Case 3] $C_1$ and $C_2$ have the same colour.
\end{description}

We introduce the notion of {contracting sets}, which will be convenient when arguing about  part~\ref{itm:robust-matching} of Lemma~\ref{lem:components}.

\begin{definition}\label{def:contracting-set}
	For a spanning subgraph $H \subset G$, we call a stable set $S \subset V(G)$ \emph{contracting} in~$H$, if $|N_H(S)|  < |S| + \gamma n$.
\end{definition}

Note that if $S$ is contracting in $H$, then $2|N_H(S)| - \gamma n  < |S| +  |N_H(S)| \leq  n$ and, therefore,
\begin{equation}\label{equ:NS>1/2}
	|N_H(S)| <(1/2 + \gamma /2)n.
\end{equation}

\subsection{\texorpdfstring{Case 1: One of $C_1$ and $C_2$ is spanning}{Case 1: One of C1 and C2 is spanning}}
\label{cas:C1-spanning}
We assume that $R:=C_1$ is spanning and red.
We will show that either there exist one or two blue components which together with $R$ satisfy the conditions of Lemma~\ref{lem:components}, or the colouring is $(4\gamma)$-extremal as in Definition~\ref{def:extremal-colouring}\ref{conf:bipartite} in contradiction to our assumption.

Let us begin with the following observation. For a collection of blue components $\{B_i, i\in I\}$,
consider a subgraph $H = R\cup \bigcup_{i\in I} B_i \subset G$, and assume that some stable set $S$ is contracting in~$H$.
If $B$ is a blue component with $v \in S \cap V(B)$, then
\begin{equation}\label{equ:S>2/3-VB}
	|S| + \gamma n > |N_H(S)| \geq \deg_H(v) \geq (2/3 +8\gamma)n - |V(B)|.
\end{equation}
Furthermore, we must have that $B \neq B_i$, $i\in I$, since otherwise $|N_H(S)| \geq \deg_H(v)=\deg_G(v)\ge (2/3 + 8\gamma) n$, in contradiction to~\eqref{equ:NS>1/2}.
Hence, for any $B_i$, $i\in I$, we have
\begin{equation}\label{equ:C-cap-S-es}
	S \cap V(B_i)  =\es.
\end{equation}

The next claim summarizes a few observations on contracting sets.
\begin{claim}\label{cla:observation-on-contracting-sets}
	Let $H = R\cup \bigcup_{i\in I} C_i$ for blue components $C_i$, $i\in I$.
	Suppose $S$ is a contracting set in $H$, and let $B_1,\ldots,B_t$ be the blue components that have non-empty intersection with $S$.
	Then
	\begin{enumerate}[label=\abc]
		\item $t \leq 2$ and $\big|V(\bigcup_{j\in [t]} B_j)\big| \geq (1/3 + 3 \gamma)n$.
	\end{enumerate}
	In addition, we have the following:
	\begin{enumerate}[label=\abc] \setcounter{enumi}{1}
		\item\label{itm:obs-contracting-sets-D-bipartite}	If $S$ intersects with a bipartite blue component, then $t=1$.
		\item \label{itm:obs-sets-t=2}	If $t=2$, then $V(G) \sm N_H(S) \subset V(\bigcup_{j\in [t]} B_j)$ and $V(G)\sm V(\bigcup_{j\in [t]}  B_j) \subset N_{H}(S)$.
	\end{enumerate}
\end{claim}
Before we prove Claim~\ref{cla:observation-on-contracting-sets} let us note that, by~\eqref{equ:C-cap-S-es}, the sets $V(B_j)$, $j\in [t]$, are distinct from the sets $V(C_i)$, $i\in I$.
\begin{claimproof}
	Set $W = V(G) \sm (N_{H}(S) \cup S)$  and note that
	\begin{equation}\label{equ:C-no-edges-in-S}
		\text{$H$ has no edges in $S$ and no edges between $S$ and $W$.}
	\end{equation}
	For each $j \in [t]$, fix a vertex  $v_j\in V(B_j) \cap S$ and observe that
	\begin{equation} \label{equ:ScupWcapB-geq-2/3-NS}
		|(S \cup W) \cap V(B_j)| \geq \deg_{B_j}(v_j,S \cup W) \overset{\eqref{equ:C-no-edges-in-S}}{\geq} (2/3  + 8 \gamma)n - |N_H(S)|.
	\end{equation}
	To prove $t \leq 2$, let us assume that $t \geq 3$ and obtain a contradiction.
Since $B_1, B_2, B_3$ are vertex disjoint, we have
	By~\eqref{equ:C-no-edges-in-S} we have
	\begin{align*}
		|S \cup W| & \geq \sum_{i \in [3]} |(S \cup W) \cap V(B_j)| \overset{\eqref{equ:ScupWcapB-geq-2/3-NS}}{\geq} (2+ 24 \gamma)n - 3|N_H(S)|.
	\end{align*}
	This, together with the fact that $V(G)=S \cup N_H(S) \cup W$, contradicts \eqref{equ:NS>1/2}:
	\begin{align*}
		3 |N_H(S)| + |S \cup W| & \overset{ }{\geq} (6/3  + 24 \gamma)n & \Leftrightarrow \nonumber \\
		2|N_H(S)|               & \geq(3/3  + 24 \gamma)n               & \Leftrightarrow \nonumber \\
		|N_H(S)|                & \geq(1/2  + 12 \gamma)n .              & \nonumber
	\end{align*}
	Hence $t \leq 2$.

	In preparation for proving parts~\ref{itm:obs-contracting-sets-D-bipartite} and~\ref{itm:obs-sets-t=2}, we show the following:
	\begin{align}\label{equ:If-S-large-then-t=2}
		\text{If $t=2$, then $|S| \geq (1/3 + 15\gamma)n$.}
	\end{align}
	Indeed, assuming that $t=2$, we have 
	\begin{align*}
	 |W| & \geq |W \cap V(B_1)| + |W \cap V(B_2)|\\
                    &\geq |(W\cup S) \cap V(B_1)| -  |S \cap V(B_1)| + |(W\cup S) \cap V(B_2)|   - |S \cap V(B_2)|\\
		    &\stackrel{\eqref{equ:ScupWcapB-geq-2/3-NS}}{\geq} 2 (2/3 + 8 \gamma)n  - 2|N_H(S)| - |S \cap (V(B_1)\cup V(B_2))|\\
		    &\geq  (4/3 + 16 \gamma)n  - 2|N_H(S)| - |S|.
	\end{align*}
        This, together with $V(G) = W \cup S \cup N_H(S)$, gives
	$|S| + \gamma n > |N_H(S)| \geq (1/3 + 16 \gamma)n$,
	which implies~\eqref{equ:If-S-large-then-t=2}.

	To obtain part~\ref{itm:obs-sets-t=2} suppose that $t=2$.
	It follows from $\delta(G)\ge (2/3 + 8\gamma) n$ and \eqref{equ:If-S-large-then-t=2} that any vertex in $V(G) \sm (N_{H}(S) \cup S)$ has an edge to~$S$. All such edges are not in $H$ by~\eqref{equ:C-no-edges-in-S}, so they are blue and in $B_1$ or $B_2$.
	Hence $V(G) \sm N_{H}(S) \subset V(\bigcup_{j\in [t]} B_j)$ and $V(G)\sm V(\bigcup_{j\in [t]} B_j) \subset N_H(S)$, as desired.

	To show part~\ref{itm:obs-contracting-sets-D-bipartite}, let us assume that $t=2$ and, without loss of generality, $B_1$ is bipartite.
	Consider any vertex $v_1 \in S\cap V(B_1)$.
	By~\eqref{equ:If-S-large-then-t=2} and since $S$ is stable in $H$, the vertex $v_1$ has a blue neighbour $v_2 \in S$.
	Clearly $v_1,v_2$ belong to distinct bipartition classes of $B_1$ and their neighbourhoods in $B_1$ are disjoint. Hence
	$$
		|(S \cup W) \cap V(B_1)| \geq \sum_{j=1}^2 \deg_{B_1}(v_j,S \cup W)| \overset{\eqref{equ:ScupWcapB-geq-2/3-NS}}{\geq} 2(2/3  + 8 \gamma)n - 2|N_H(S)|.
	$$
       We combine the last inequality with \eqref{equ:ScupWcapB-geq-2/3-NS} and obtain
	\begin{align*}
		|S \cup W| & \geq \sum_{i \in [2]} |(S \cup W) \cap V(B_j)| \geq (2+ 24 \gamma)n - 3|N_H(S)|.
	\end{align*}
	Using  the fact that $V(G)=S \cup N_H(S) \cup W$, we get again
	\begin{align*}
		3 |N_H(S)| + |S \cup W|  \overset{}  {\geq} ( {6}/{3}  + 24 \gamma)n  \overset{\eqref{equ:ScupWcapB-geq-2/3-NS}}  {\Rightarrow} 
		|N_H(S)|                             \geq( {1}/{2}  + 12 \gamma)n,                                                
	\end{align*}
	which is a contradiction to~\eqref{equ:NS>1/2}.
	This proves part~\ref{itm:obs-contracting-sets-D-bipartite}.

	Finally, to complete the proof of part (a), we must show that $|V(\bigcup_{j\in [t]} B_j)| \geq (1/3 + 3\gamma )n$.
	If $|S| \geq (1/3 +3 \gamma) n$, this follows from $S \subset V(\bigcup_{j\in [t]} B_j)$.
	Otherwise, we have $|S| < (1/3 + 3 \gamma) n$, and hence $t=1$ by~\eqref{equ:If-S-large-then-t=2}.
	Then
	$$
		(1/3 + 4 \gamma )n > |S| + \gamma n  \overset{\eqref{equ:S>2/3-VB}}{>} (2/3 +8\gamma)n - |V(B_1)|
	$$
	from which we obtain $|V(B_1)|  > (1/3 + 4 \gamma n)$ as desired.
\end{claimproof}

Now we show that we can select two blue components which together with $R$ satisfy the conditions of Lemma~\ref{lem:components}.
\begin{remark}\label{rem:connectivity-is-trivial}
	Note that since we always select $R$ and $R$ is spanning, part \ref{itm:connectivity} of Lemma~\ref{lem:components} holds trivially.
\end{remark}

\begin{claim}\label{cla:no-2-large-blue-cps}
	If there are two blue components $B_1,B_2$ with $|V(B_1)|,|V(B_2)| \geq (1/3 + 3 \gamma)n$, then  Lemma~\ref{lem:components} is true.
\end{claim}
\begin{claimproof}
	Consider  blue components  $B_1,B_2$ with $|V(B_1)|,|V(B_2)| \geq (1/3 +3 \gamma)n$.
	First, suppose that one of $R,B_1,B_2$ contains an odd cycle.
	Let $H_2  = R \cup B_1  \cup B_2$.
	We will show that $H_2$ satisfies the conditions of Lemma~\ref{lem:components}.
	If $H_2$ has a contracting set, then by  
	Claim~\ref{cla:observation-on-contracting-sets} there are (up to) two more blue 
	components which together contain at least $(1/3 + 3 \gamma)n$ vertices.
	But this contradicts $|V(B_1)|,|V(B_2)| \geq (1/3 +3 \gamma)n.$
	Hence $H_2$ satisfies Lemma~\ref{lem:components}\ref{itm:robust-matching}.
	Moreover, $H_2$ also trivially satisfies~\ref{itm:overlap}\ref{itm:true-overlap}.
	Since one of $R,B_1,B_2$ contains an odd cycle, $H_2$ satisfies Lemma~\ref{lem:components}\ref{itm:odd-monochromatic-cycle}, and we are done.

	Hence we can assume that $R,B_1,B_2$ are each bipartite. Denote the bipartite 
	partition classes of $R$ by $X$ and $Y$, where $|X| \geq |Y|$. We claim that 
	either $V(B_1)=X$ and $V(B_2)=Y$ or vice versa. Observe that 
	$\delta(G_\b[X]) \geq (2/3 + 8 \gamma) n-|Y|\geq (1/6+8\gamma )n$. Since 
	$B_1$ and $B_2$ have each order at least $(1/3+3\gamma)n$ and $|Y|\leq n/2$, 
	there must be a vertex $x \in X$ which is also in $B_1$ or $B_2$, say in $B_1$.

	Let us show that $X$ is contained in~$V(B_1)$. Since 
	$\delta(G_\b[X]) \geq (1/6+8\gamma )n$, the component $B_1$ has an edge 
	in $X$. It follows, using $\delta(G_\b[X]) \geq  (2/3 + 8 \gamma) n-|Y|$ and the 
	bipartiteness of $B_1$, that $B_1$ has at least $(4/3+16\gamma )n-2|Y|$ 
	vertices in $X$.
	Take any vertex $x' \in X$. Since
	$$(4/3+16\gamma )n-2|Y| + \delta(G_\b[X]) \ge  (2 + 24 \gamma) n-3|Y|>|X|,$$
	$x'$ must have a blue neighbour in $V(B_1)$.
	This proves  $X \subset V(B_1)$.

	Since $X$ is contained in $V(B_1)$, it follows that $V(B_2)$ is contained in $Y$.
	In fact, we have $V(B_2)=Y$. To see this, recall that 
	$|V(B_2)| \geq (1/3 + 3 \gamma)n$. As $G$ has minimum degree 
	$(2/3+8\gamma)n$ and $Y$ contains no red edges, every vertex $y \in Y$ 
	must have a blue neighbour in $V(B_2)$, and therefore also belongs to $B_2$.
	Hence $Y \subset V(B_2)$.	So in fact, we have $V(B_1)=X$ and $V(B_2)=Y$.

	Set  $H_1 = R \cup B_{1}$.	If $H_1$ has no contracting sets, then it satisfies the 
	conditions of Lemma~\ref{lem:components} and we are done.
	So suppose that $S$ is a contracting set in $H_1$.
	Since vertices in $X$ have degree at least $(2/3+8\gamma)n$ in $H_1$, the 
	set $S$ must be contained in $Y=V(B_2)$.
	In the following, we distinguish two cases depending on the size of $S$.

	Let us begin with the case, where $|S| > (1/3-8\gamma)n$.
	Then every vertex of $G$ has a neighbour in~$S$. In particular, every vertex in $X$ 
	has a red neighbour in $S$, hence $X\subset N_{H_1}(X)$.
	So $|Y| + \gamma n \geq |S| + \gamma n \geq |N_{H_1}(S)| \geq |X|$, and 
	hence $|X|\geq |Y|\geq |X| -\gamma n$. It follows that the colouring is 
	$\gamma$-extremal as in~Definition~\ref{def:extremal-colouring}\ref{conf:bipartite}, 
	which contradicts our initial assumption.

	Now suppose that $|S|  \leq (1/3-8\gamma)n$. We show that this results a 
	contradiction. Suppose first that  $G[S]$ contains an edge $vw$. Such an edge 
	must be blue and in $B_2$. Using the bipartiteness  of $B_2$, we obtain that
	\begin{align*}
		|V(B_2)| & \overset{}{\geq} \deg_G(v,V(G)\sm X) + \deg_G(w,V(G)\sm X) \\
		         & \overset{}{\geq} (4/3 + 16 \gamma) n -  2|N_{H_1}(S)|                    \\
		         & \geq (4/3 + 14 \gamma) n -2|S|                                           \\
		         & \geq (2/3 + 30 \gamma)n,
	\end{align*}
	where we used that $S$ is contracting in the penultimate line.
	Now suppose that $G[S]$ contains no edge.
	It follows, for any $v \in S$, that
	\begin{align*}
		|V(B_2)| & \overset{}{\geq} |S|+ \deg_G(v,V(G)\sm N_{H_1}(S)) \\
		         & \geq    |S|+ (2/3 + 8 \gamma) n - |N_{H_1}(S)|     \\
		         & \geq (2/3 + 7 \gamma)n.
	\end{align*}
	In both cases, we have $|V(B_2)| \geq 2n/3$, which contradicts $|V(B_1)| \geq (1/3+3\gamma)n$.
\end{claimproof}

Denote the three blue components of largest order by $B_1,B_2,B_3$, where $|B_1| \geq |B_2| \geq |B_3|$.
Let $H_1=R\cup R_1$ and, for $i \in \{2,3\}$, let $H_i = R \cup B_1 \cup B_i = H_1 \cup B_i$.

\begin{claim}\label{cla:H0-no-contracting-set}
	Let $i \in \{2,3\}$. If $H_i$ has no contracting set, then  Lemma~\ref{lem:components} holds.
\end{claim}
\begin{claimproof}
	Suppose that $H_i$ has no contracting set for some  $i \in \{2,3\}$. 
	
	First we deal with the case, when $|V(B_1) \cup V(B_2)| <	(1/3 + \gamma)n$.
	Here, we shall show that $H_1$ satisfies  the conditions of  
	Lemma~\ref{lem:components}. To begin, observe that all blue components 
	distinct from $B_1$ have order at most $(1/6 + \gamma/2)n$.
	It follows that $H_1$ has minimum degree at least  
	$\delta(G)-(1/6 + \gamma/2)n>(1/2 + 7\gamma)n$ and, therefore, contains 
	no contracting sets (see \eqref{equ:NS>1/2}).
	So~$H_1$ satisfies Lemma~\ref{lem:components}\ref{itm:robust-matching}.
	If $R$ contains an odd cycle, then $H_1$ also satisfies Lemma~\ref{lem:components}\ref{itm:overlap}\ref{itm:spanning-and-empty} and~\ref{itm:odd-monochromatic-cycle}, and we are done.
	Thus we may assume that $R$ is bipartite.
	Let $V(R) = X \cup Y$ be a bipartition of $R$, where $|X| \geq |Y|$.
	Note that $\delta_\b(G[X]) \geq (2/3 + 8\gamma )n - |Y| \geq (1/6 + 8\gamma )n$.
	Since $B_1$ is the only blue component of order (possibly) more than 
	$(1/6 + \gamma)n$, it follows that $X\subset V(B_1)$. Thus 
	$|V(B_1)| \geq |X| \geq n/2$ in contradiction to the assumption of 
	$|V(B_1) \cup V(B_2)| <(1/3 + \gamma)n$. 
	
	Hence we may assume that  $|V(B_1) \cup V(B_2)| \geq (1/3 + \gamma)n$.  
	This means that $H_2$ satisfies 
	Lemma~\ref{lem:components}\ref{itm:overlap}\ref{itm:true-overlap}. If 
	$|V(B_1) \cup V(B_3)| \geq (1/3 + \gamma)n$, then also $H_3$ satisfies 
	Lemma~\ref{lem:components}\ref{itm:overlap}\ref{itm:true-overlap}. Otherwise, 
	$|V(B_1) \cup V(B_3)| <	(1/3 + \gamma)n$ and all blue components 
	distinct from $B_1$ and $B_2$ have order at most $(1/6 + \gamma/2)n$.
	It follows that $H_2$ has minimum degree at least  
	$\delta(G)-(1/6 + \gamma/2)n>(1/2 + 7\gamma)n$ and, therefore, contains 
	no contracting sets (see \eqref{equ:NS>1/2}). So, we may assume $i=2$. 
	In either case, $H_i$ satisfies 
	Lemma~\ref{lem:components}\ref{itm:overlap}\ref{itm:true-overlap}.

	If one of $R$, $B_1$ and $B_i$ is not bipartite, then $H_i$ satisfies 
	Lemma~\ref{lem:components}\ref{itm:odd-monochromatic-cycle}, and
	Lemma~\ref{lem:components} holds.
	Hence assume  that $R$, $B_1$ and $B_i$ are each bipartite.
	As before, let $V(R) = X \cup Y$ be a bipartition of $R$ with $|X| \geq |Y|$.

	We need to distinguish between two sub-cases: either $B_1\cap X\neq\es$ 
	or $V(B_1) \subset Y$. Suppose first that $B_1$ has at least one vertex in $X$.
	We will show that  $X \subset V(B_1)$.
	To this end, note that $\delta_\b(G[X]) \geq \delta(G)-|Y|\ge (1/6 + 8\gamma )n$.
	Since $B_1$ is bipartite and has at least one edge in $X$, it has
	at least $2\delta_\b(G[X])\ge (1/3 + 16\gamma )n$ vertices in $X$.
	Using that $\delta(G) \geq (2/3+8\gamma)n$, we obtain that any vertex $x$ in $X$ 
	has a neighbour $x'$ in $V(B_1)\cap X$. Since $R$ is  bipartite, $xx'$ must be blue, 
	hence $x\in V(B_1)$ and $X \subset V(B_1)$.

	The situation, where $|X| < (1/2 + 2\gamma)n$ is easily resolved. Indeed, 
	if $|X| < (1/2 + \gamma)n$, then $|V(B_{i})|  \geq 2\delta(G_\b[Y]) \geq 2((2/3 + 8\gamma )n -|X|) \geq (1/3 + 12\gamma )n$, where we used the bipartiteness of $B_{i}$.
	Since $|V(B_1)| \geq |V(B_i)|$, we are done by Claim~\ref{cla:no-2-large-blue-cps}.
	
	Thus we can assume that $|X| \geq (1/2 + 2\gamma)n$. Now it is not hard to see 
	that $H_1 = R \cup B_1$ satisfies the conditions of Lemma~\ref{lem:components}.
	Note that~\ref{itm:overlap}\ref{itm:true-overlap} and~\ref{itm:odd-monochromatic-cycle} are trivially satisfied as $X\subset V(B_1)$.
	If $H_{1}$ has no contracting sets, then~\ref{itm:robust-matching} also holds, and we are done.
	So assume that $S$ is a contracting set in $H_{1}$. By 
	Claim~\ref{cla:observation-on-contracting-sets} there are blue components 
	$Q_1,\dots,Q_t$ with $t\leq 2$ such that
	 $|V(\bigcup_{j\in[t]} Q_j)| \geq (1/3 + 3 \gamma)n$.
	Note that $t = 2$, since otherwise $|V(B_1)| \geq |V(Q_1)| \geq (1/3 + 3\gamma)n$ 
	and we are done by Claim~\ref{cla:no-2-large-blue-cps} again.
	So, by Claim~\ref{cla:observation-on-contracting-sets}\ref{itm:obs-sets-t=2}, we 
	have $X\subset V(B_1) \subset N_{H_1}(S)$.
	Together with $|X| \geq (1/2 + 2\gamma)n$, this contradicts~\eqref{equ:NS>1/2}.
	This concludes the analysis of the case when $B_1$ intersects with $X$.

	Now suppose that $V(B_1) \subset Y$.
	We claim that no bipartite blue component has a vertex in $X$.
	Indeed, assume otherwise.
	Then such a component has order at least $(1/3 + 16\gamma )n$ due to being bipartite and $\delta(G_\b[X]) \geq (2/3 + 8\gamma )n - |Y| \geq (1/6 + 8\gamma )n$.
	Hence $|V(B_1)| \geq (1/3 + 16\gamma )n$ by the maximality of $B_1$, and we are done again by Claim~\ref{cla:no-2-large-blue-cps}.
	So blue components with vertices in $X$ are not bipartite.
	In particular, $V(B_i) \subset Y \sm V(B_1)$.

	Next, we show that $B_{5-i}$ has a vertex in $X$.
	Indeed, assume otherwise.
	Owing to $\delta(G_\b[X]) \geq (1/6 + 8\gamma )n$ and $B_3$ being the third largest component, it follows that $|V(B_1)| \geq |V(B_2)| \geq |V(B_3)| \geq (1/6 + 8\gamma )n$.
	Hence $|Y| \geq |V(B_1)|+|V(B_2)|+|V(B_3)| > n/2$, which contradicts $|X| \geq |Y|$.
	So $B_{5-i}$ intersects with $X$, which implies in particular that $B_{5-i}$ is not bipartite.

	To finish, we show that $H_{5-i}=R \cup B_1 \cup B_{5-i}$ satisfies the conditions of Lemma~\ref{lem:components}.
	Note that $|V(B_{5-i})|  \geq (1/6 + 8\gamma )n$ as $\delta(G_\b[X]) \geq (1/6 + 8\gamma )n$.
	Hence $|V(B_1) \cup V(B_{5-i})| \geq (1/3 + 16\gamma )n$, which gives~\ref{itm:overlap}\ref{itm:true-overlap}.
	Moreover,~\ref{itm:odd-monochromatic-cycle} holds as  $B_{5-i}$ is not bipartite.
	If $H_{5-i}$ has no contracting sets, then we also have the 
	condition~\ref{itm:robust-matching}.
	Hence assume that $S$ is a contracting set in $H_{5-i}$.

	By Claim~\ref{cla:observation-on-contracting-sets} applied to $H_{5-i}$, there are blue components $Q_1,\dots,Q_t$ with $t\leq 2$ such that $|V(\bigcup_{j\in[t]} Q_j)|  \geq (1/3 + 3 \gamma)n$.
	As before, we observe that $t = 2$, since otherwise $|V(B_1)| \geq |V(Q_1)| \geq (1/3 + 3\gamma)n$, and we are again done by Claim~\ref{cla:no-2-large-blue-cps}.
	As $B_i$ is bipartite, it follows that $B_i \neq Q_1,Q_2$ by Claim~\ref{cla:observation-on-contracting-sets}\ref{itm:obs-contracting-sets-D-bipartite}.
	Moreover, by Claim~\ref{cla:observation-on-contracting-sets}\ref{itm:obs-sets-t=2}, we have $V(B_1) \cup V(B_2) \cup V(B_3)= V(B_1) \cup V(B_{5-i}) \cup V(B_i) \subset N_{H_5-i}(S)$.
	On the other hand, $|V(B_3)| \geq  (1/6 +  \gamma)n$ since one of $Q_1$ and $Q_2$ has size at least $(1/6 +  \gamma)n$ and $B_3$ is the  blue component of third largest order.
	Hence, $|N_{H_{5-i}}(S)| \geq |V(B_1 \cup B_2 \cup B_3)| \geq (1/2 +  3\gamma)n$ in contradiction to~\eqref{equ:NS>1/2}.
\end{claimproof}

By Claim~\ref{cla:H0-no-contracting-set}, we can assume that $H_2$ and $H_3$ have each a contracting set, which we denote by $S_2$ and $S_3$, respectively.

\begin{claim}\label{cla:Si>1/2}
	We have $|S_2|, |S_3| \geq (1/2 - \gamma)n$.
\end{claim}
\begin{claimproof}
	Fix $i \in \{2,3\}$.
	By Claim~\ref{cla:observation-on-contracting-sets} applied to $H_i$, there are blue components $Q_1,\dots,Q_t$ for $t\leq 2$, which together cover at least $(1/3 + 3 \gamma)n$ vertices.
	If $t=1$, then $Q_1$ has order at least $(1/3 + 3 \gamma)n$.
	In this case, $B_1$ (as the largest blue component) has also order at least $(1/3 + 3 \gamma)n$, and we are done by Claim~\ref{cla:no-2-large-blue-cps}.
	Hence we have $t=2$.
	Now Claim~\ref{cla:observation-on-contracting-sets}\ref{itm:obs-sets-t=2} yields $V(G)\sm N_{H_i}(S_i) \subset V(\bigcup_{j\in [t]} Q_j)$ and, therefore, we must have $V(B_1 \cup  B_i) \subset N_{H_i}(S_i)$.
	Since $B_1, B_2, B_3$ are three blue components of largest order and $|V(B_1)| \geq |V(B_2)| \geq |V(B_3)|$, we have $|V(Q_1)| + |V(Q_2)| \leq |V(B_1)| + |V(B_i)|$.
	Together, this gives
	\begin{align*}
		n-|S_i| -\gamma n & \leq |V(G) \sm  N_{H_i}(S_i)|
		\leq|V(Q_1)| + |V(Q_2)|
		\leq |V(B_1)| + |V(B_i)|
		\\&\leq |N_{H_i}(S_i)|
		\leq |S_i| +\gamma n.
	\end{align*}
	Hence  $|S_i| \geq (1/2-\gamma)n$.
\end{claimproof}

\begin{claim}\label{cla:size-B2-intersection-S3}
	We have $|V(B_2)\cap S_3| ,~  |V(B_3)\cap S_2| \geq (1/4- 4\gamma)n$.
\end{claim}
\begin{claimproof}
	Since $|S_2| \geq (1/2 - \gamma)n$ by Claim~\ref{cla:Si>1/2}, the vertices in $S_2$ have at least $(1/6+7\gamma)n$ blue neighbours in $S_2$.
	It follows from Claim~\ref{cla:observation-on-contracting-sets} that $S_2$ is contained in at most two blue components, one of which has order at least $(1/4-\gamma)n$.
	Thus $|V(B_1)| \geq |V(B_2)| \geq |V(B_3)| \geq (1/4-\gamma)n$ by the maximality of $B_1,B_2,B_3$.
	Recall that $H_2 = R \cup B_1 \cup B_2$ and, therefore, $S_2 \cap B_1 =\es$ and $S_2 \cap B_2 =\es$ by~\eqref{equ:C-cap-S-es}.
	As $|S_2| \geq (1/2 - \gamma)n$, $ |V(B_1)\cup V(B_2)| \geq (1/2-2\gamma)n$ and $V(B_1) \cup V(B_2) \subset V(G) \sm S_2$, the component $B_3$ has at most $3\gamma n$ vertices outside of $S_2$.
	Consequently, $|V(B_3)\cap S_2| \geq (1/4-4\gamma)n$.

	Analogously, as $|S_3| \geq (1/2 - \gamma)n$ by Claim~\ref{cla:Si>1/2} and $V(B_1) \cup V(B_3) \subset V(G) \sm S_3$, the component~$B_2$ has at most $3\gamma n$ vertices outside of $S_3$.
	As a result, we have $|V(B_2)\cap S_3| \geq (1/4-4\gamma)n$.
\end{claimproof}

Note that $V(B_1)$ has vertices  in  neither $S_2$ nor $S_3$ by \eqref{equ:C-cap-S-es} and $|V(B_1)|\ge |V(B_2)|\ge |V(B_2)\cap S_3| \geq (1/4-4\gamma)n$.
So, by Claim~\ref{cla:Si>1/2}, there is a vertex $v \in S_2 \cap S_3$.
By~\eqref{equ:C-cap-S-es} and the definition of $H_2$ and $H_3$, the vertex $v$ is not in $V(B_2 \cup B_3)$.
At the same time, $v$ has no red neighbours in $S_2 \cup S_3$ since $S_2, S_3$ are stable in $H_2,H_3$ respectively.
Therefore,~$v$ has no neighbours in $S_2 \cap V(B_3)$ and no neighbours $S_3 \cap V(B_2)$.
Since $B_2$ and $B_3$ are distinct blue components, these sets are disjoint.
Moreover, by Claim~\ref{cla:size-B2-intersection-S3}, $B_2$ and $B_3$  have each size  at least $(1/4-4\gamma)n$.
Hence $v$ has degree at most $n-2(1/4-4\gamma)n=(1/2+8\gamma)n$ in contradiction with $\delta(G) \geq (2/3 + 8\gamma)n$.
This concludes the analysis of the case, where one of $C_1$ and $C_2$ is spanning.

\subsection{\texorpdfstring{Case 2: $C_1$ and $C_2$ have distinct colours}{Case 2: C1 and C2 have distinct colours}}
\label{cas:C1-C2-distinct-colours}
Suppose that $R:=C_1$ is red and $B:=C_2$ is blue.
If $R$ or $B$ is spanning, we proceed as in Subsection~\ref{cas:C1-spanning}.
We will show that $R$, $B$ and another red component together satisfy the conditions of Lemma~\ref{lem:components}.

Since neither $R$ nor $B$ is spanning, there are vertices $v_R \in  V(R) \sm V(B)$ and $v_B \in  V(B) \sm V(R)$.
Note that
\begin{equation}\label{equ:Nv-Nw-geq-1/3}
	|N_G(v_R) \cap N_G(v_B)| \geq (1/3 + 16\gamma)n.
\end{equation}
Moreover,  any common neighbour of $v_R$ and $v_B$ must lie in $V(R)\cap V(B)$ (if for instance, $v_R$ has a neighbour $v\in V(B)\setminus V(R)$ then the edge $v_Rv$ cannot be red and it cannot be blue, a contradiction). It follows that
\begin{equation}\label{eq:intersecV1V2}
	|V(R)\cap V(B)| \geq (1/3 + 16\gamma)n.
\end{equation}
Similarly, observe that, for $i \in \{1,2\}$,
\begin{equation}\label{equ:V(C_i)-sm-V(C_j)<1/3-gamma}
	|V(R) \sm V(B)| \leq |V(G) \sm N_G(v_{B})| < (1/3 - 8\gamma)n
\end{equation}
and hence $|V(B)| \geq (2/3+8\gamma)n$.
By symmetry, the same is true for colours flipped.

Let $R'$ be the red component of largest order in $B - V(R)$.
Note that any vertex $v'_B$ in $B - V(R \cup R')$ is incident to fewer than $(1/6 - 4\gamma )n$ red edges.
Indeed, otherwise $|V(R')| \ge (1/6 - 4\gamma )n$ (by maximality of $R'$), and so $v'_B$ sends a red edge to $R'$  by~\eqref{equ:V(C_i)-sm-V(C_j)<1/3-gamma} (with colours flipped).

We claim that $H := R \cup B \cup R'$ satisfies items~\ref{itm:robust-matching}--\ref{itm:connectivity} of Lemma~\ref{lem:components}.
As by Claim~\ref{cla:cover-by-C1-C2}, $R \cup B$ spans $G$, part~\ref{itm:connectivity} is trivial and  part~\ref{itm:overlap} holds because of~\eqref{eq:intersecV1V2}.
By~\eqref{equ:Nv-Nw-geq-1/3}, and because of our assumption on the minimum degree of $G$, we know that $G[N(v_R) \cap N(v_B)]$ contains an edge. If this edge is red,
$v_R$ is in a red triangle, and if it is blue, $v_B$ is in a blue triangle.
Hence part~\ref{itm:odd-monochromatic-cycle} follows.

It remains to show that  part~\ref{itm:robust-matching} is satisfied.
To this end, we assume there is a contracting set~$S$ in $H$ and show that this results in a contradiction.
Observe that
any vertex $v$ in $V(R) \cap V(B)$ or in $V(R')$ has degree
\begin{equation*}
	\deg_H(v) \geq (2/3 + 8 \gamma)n.
\end{equation*}
Also, recall that by definition of $R'$, any vertex $v'_B \in V(B) \sm V(R \cup  R')$ has $\deg_\r(v'_B) < (1/6 -4\gamma)n$ and so
\begin{equation*}\label{equ:deg-v-geq-2/3}
	\deg_H(v'_B) > (2/3 + 8 \gamma)n - (1/6 -4\gamma)n = (1/2 + 12 \gamma)n.
\end{equation*}
On the other hand, since $S$ is contracting, we have $|N_H(S)| \leq (1/2 + \gamma /2)n$ by~\eqref{equ:NS>1/2}.
Since $|N_H(S)| \geq \deg_H(v)$ for any $v \in S$, this means that $S \subset V(R) \sm V(B)$.
By~\eqref{equ:V(C_i)-sm-V(C_j)<1/3-gamma} this gives $|S| \leq (1/3 - 8 \gamma)n$.
However, any vertex $v'_R \in V(R) \sm V(B)$ has
\begin{equation*}
	\deg_H(v'_R) \geq (2/3 + 8 \gamma)n - |V(R) \sm V(B)| \overset{\eqref{equ:V(C_i)-sm-V(C_j)<1/3-gamma}}{\geq} (1/3 + 16 \gamma)n.
\end{equation*}
Thus $|N_H(S)| \geq |S| + 24 \gamma n$. So, contrary to our assumption, $S$ is not contracting.

\subsection{\texorpdfstring{Case 3: $C_1$ and $C_2$ have the same colour}{Case 3: C1 and C2 have the same colour}}
Suppose that $R_1:=C_1$ and $R_2:=C_2$ are both red with
\begin{equation}\label{equ:VR1-geq-VR2}
	|V(R_1)|\geq n/2\ge  |V(R_2)|.
\end{equation}
Let $B_1, \ldots, B_t$ be the blue components that intersect with $R_2$.
If $t = 1$, or if one of the $B_i$ contains~$R_1$, we proceed as in Subsection~\ref{cas:C1-C2-distinct-colours}.
Hence we can assume that $t \geq 2$, and $R_1\not\subseteq B_i$, for $1\le i\le t$.

We begin by showing that $t=2$.
To this end, let $v_i \in V(B_i) \cap V(R_2)$, for $i \in [t]$. If $t\ge3$, then (using~\eqref{equ:VR1-geq-VR2} for the third inequality),
\begin{align*}
	|V(R_1)|  \ge \sum_{i \in [3]} \deg(v_i,V(R_1)) & \geq 3\big((2/3 + 8 \gamma)n - |V(R_2)|\big) \\
	                                                & >2n-(n+|V(R_2)|)                             \\ & = n- |V(R_2)|,
\end{align*}
a contradiction.
So indeed $t=2$.

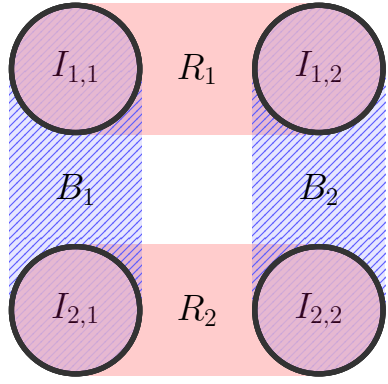
\begin{figure}
	\centering
	\begin{tikzpicture}[scale=0.8, transform shape]

		\node[white set] (a) at (0,0) {\huge $I_{2,1}$};
		\node[set,fill opacity=0.2, fill=red] (a) at (0,0) {};
		\node[preaction={pattern=north east lines, opacity=0.4,
					pattern color=blue!60}, edge, set,fill opacity=0.1, fill=blue] (a) at (0,0) {};

		\node[white set] (b) at (4,0) {\huge $I_{2,2}$};
		\node[set,fill opacity=0.2, fill=red] (b) at (4,0) {};
		\node[preaction={pattern=north east lines, opacity=0.4,
					pattern color=blue!60}, edge, set,fill opacity=0.1, fill=blue] (b) at (4,0) {};

		\node[white set] (c) at (0,4) {\huge $I_{1,1}$};
		\node[set,fill opacity=0.2, fill=red] (c) at (0,4) {};
		\node[preaction={pattern=north east lines, opacity=0.4,
					pattern color=blue!60}, edge, set,fill opacity=0.1, fill=blue] (c) at (0,4) {};

		\node[white set] (d) at (4,4) {\huge $I_{1,2}$};
		\node[set,fill opacity=0.2, fill=red] (d) at (4,4) {};
		\node[preaction={pattern=north east lines, opacity=0.4,
					pattern color=blue!60}, edge, set,fill opacity=0.1, fill=blue] (d) at (4,4) {};


		\node[] (x) at (2,0) {\huge $R_2$};
		\node[] (x) at (2,4) {\huge $R_1$};
		\node[] (x) at (0,2) {\huge $B_1$};
		\node[] (x) at (4,2) {\huge $B_2$};

		\begin{scope}[on background layer]
			\fill[red edge, opacity=0.2] (a.south) -- (b.south) -- (b.north) -- (a.north) -- (a.south) --  cycle;
			\fill[red edge, opacity=0.2] (c.south) -- (d.south) -- (d.north) -- (c.north) -- (c.south) --  cycle;
			\fill[blue edge, opacity=0.1] (a.west) -- (c.west) -- (c.east) -- (a.east) -- (a.west) --  cycle;
			\fill[blue edge,opacity=0.1] (b.west) -- (d.west) -- (d.east) -- (b.east) -- (b.west) --  cycle;
		\end{scope}

	\end{tikzpicture}
	\caption{The partition of $G$ in Case 3.}
	\label{fig:case-3}
\end{figure}

Now let us partition $V(G)$ into non-empty sets
$$\I{i}{j} := V(R_i) \cap V(B_j)$$ for $i,j \in [2]$ as illustrated in Figure~\ref{fig:case-3}.
Since vertices in $ I_{i,j}$ have no neighbours in $I_{3-i,3-j}$ and as $\delta(G) \geq (2/3 + 8\gamma)n$, we obtain
\begin{equation}\label{equ:I<1/3}
	|\I{i}{j}| < (1/3 - 8\gamma)n \text{ for $i,j \in [2]$}.
\end{equation}
In particular,
\begin{equation}\label{equ:Ri,Bi>1/3}
	\text{each of $R_1,R_2,B_1,B_2$ has at least $(1/3 + 16\gamma)n$ vertices,}
\end{equation}
and moreover,
\begin{align}
	\delta\big(G[I_{i,j}] \big) & \geq \delta(G) - |I_{{3-i},j} \cup I_{i,{3-j}}| \nonumber \\ & \overset{\eqref{equ:I<1/3}}{\geq}(2/3 +8\gamma)n -2(1/3-8\gamma)n \geq   24\gamma n.\label{equ:dense-Iij-min-deg}
\end{align}
We will show that  three components among $R_1,R_2,B_1,B_2$ can be selected so that they satisfy the conditions of Lemma~\ref{lem:components} unless the colouring is $(4\gamma)$-extremal as in Definition~\ref{def:extremal-colouring}\ref{conf:four-cycle} in contradiction to our assumption.
Note that for any such choice, Lemma~\ref{lem:components}\ref{itm:connectivity} holds since $R_1 \cup R_2$, $B_1 \cup B_2$ span $G$ and part~\ref{itm:overlap}\ref{itm:true-overlap} follows  by~\eqref{equ:Ri,Bi>1/3}.
The next claim yields part~\ref{itm:odd-monochromatic-cycle}.
\begin{claim}
	If one colour has a bipartite component, then the other colour has no bipartite component.
\end{claim}
\begin{claimproof}
	Suppose that there is a bipartite component, say $R_1$.
	We claim that
	\begin{equation}\label{bipclaR1}
		\text{the bipartition classes of $R_1$ are $I_{1,1}$ and $I_{1,2}$.}
	\end{equation}
	Indeed, otherwise one of the two sets, say $I_{1,1}$, contains vertices $x, y$ from different bipartition classes.
	Clearly, the neighbourhoods of $x$ and $y$ in $R_1$ are disjoint.
	Note that the neighbours of $x$ and $y$ in $G$ are in $I_{1,2} \cup B_1$.
	Thus, each of $x$ and $y$ has at least $(2/3 + 8\gamma)n - |V(B_1)|$ neighbours in  $I_{1,2}$.
	Thus by~\eqref{equ:I<1/3},
	\[
		2\big((2/3 + 8\gamma)n - |V(B_1)|\big)  \leq |I_{1,2}| < (1/3 - 8\gamma) n,
	\]
	implying that
	$|V(B_1)|> n/2$. Applying the same argument to neighbours $x', y'$ of $x, y$ in $I_{1,2}$, we obtain that $|V(B_2)|> n/2$ as well. This is a contradiction, which proves~\eqref{bipclaR1}.

	If, $B_1$ say, is bipartite as well, then the same reasoning shows that one of its bipartition classes is $I_{1,1} $.
	Consequently  $I_{1,1}$ contains no edges, in contradiction to~\eqref{equ:dense-Iij-min-deg}.
	This proves the claim.
\end{claimproof}
It remains to show that either there is a choice of three components among $R_1, R_2, B_1,B_2$ satisfying Lemma~\ref{lem:components}\ref{itm:robust-matching} or the colouring is $(4\gamma)$-extremal as in Definition~\ref{def:extremal-colouring}\ref{conf:four-cycle}.
To this end, assume that no choice of three components among $R_1,R_2,B_1,B_2$ fulfils Lemma~\ref{lem:components}\ref{itm:robust-matching}.

For each $i \in [2]$, let $S_{R_i}$ be a contracting set of minimum size in $G-E(R_i)$, and let $S_{B_i}$ be a contracting set of minimum size in $G-E(B_i)$.
If $v \in V(R_{3-i})$ then $\deg_G(v) = \deg_{G-E(R_i)}(v)$, and so, $v \notin S_{R_i}$ by~\eqref{equ:NS>1/2}.
It follows by symmetry that for $i \in [2]$,
\begin{equation}\label{equ:SRi-subs-VRi}
	\text{$S_{R_i} \subset V(R_i)$ and $S_{B_i} \subset V(B_i)$.}
\end{equation}
In fact, we can show slightly more.
\begin{claim}\label{cla:SRi-subset-Ijj'}
	For each $i \in [2]$ there is $j \in [2]$ such $S_{R_i}\subset \I{i}{j}$.
	Similarly, for each $i \in [2]$ there is $j \in [2]$ such that $S_{B_i}\subset \I{j}{i}$.
\end{claim}
\begin{claimproof}
	Suppose that the claim does not hold for, say $S_{R_2}$.
	So, by~\eqref{equ:SRi-subs-VRi}, $S_{R_2} \cap V(B_i) \neq \es$ for $i = 1,2$.
	Then
	\begin{align*}
		|S_{R_2} \cap V(B_1)|  + | & S_{R_2} \cap V(B_2)|  + 2\gamma n  = |S_{R_2}| +  2\gamma n                                            \\
		                           & > |N_{G-E(R_2)}(S_{R_2})| +\gamma n                                                                    \\
		                           & > |N_{G-E(R_2)}(S_{R_2})|                                                                              \\
		                           & \geq |N_{G-E(R_2)}\left(S_{R_2}\cap V(B_1) \right)| + |N_{G-E(R_2)}\left( S_{R_2}\cap V(B_2) \right)|,
	\end{align*}
	where the last line follows because the red edges incident with $S_{R_2}$ belong to $R_2$ and are therefore not present in $G-E(R_2)$.
	And, moreover, there are no blue edges between $V(B_1)$ and $V(B_2)$.
	Thus by Definition~\ref{def:contracting-set}, at least one of $S_{R_2}\cap V(B_1)$, $ S_{R_2}\cap V(B_2)$ is a contracting set, in contradiction to the minimality of $S_{R_2}$.
\end{claimproof}
\begin{claim}\label{cla:SR2capB1-neq-es->I11-subs-NSRS-cap-B1}
	Let $i,j \in [2]$.
	If $S_{R_i} \cap V(B_j) \neq \es$, then $\I{3-i}{j} \subset N_{G-E(R_i)}(S_{R_i})$.
	Similarly, if $S_{B_i} \cap V(R_j) \neq \es$, then $\I{j}{3-i} \subset N_{G-E(B_i)}(S_{B_i})$.
\end{claim}
\begin{claimproof}
	Suppose the claim is wrong for $S_{R_2}$ and $B_1$.
	Let $v \in S_{R_2} \cap V(B_1)$ and $w \in \I{1}{1} \sm  N_{G-E(R_2)}(S_{R_2})$.
	Then by Claim~\ref{cla:SRi-subset-Ijj'},
	$
		S_{R_2} \subset \I{2}{1}.
	$
	Since $S_{R_2}$ is contracting in $G-E(R_2)$, we have
	\[
		|S_{R_2}|+ \gamma n  > |N_{G-E(R_2)}(S_{R_2})| \geq \deg_{G-E(R_2)}(v) \geq (2/3 + 8 \gamma) n - |V(R_2)|.
	\]
	Moreover, as $w$ has its neighbours in $G$ in $I_{2,1} \cup R_1$,
	$$
		|\I{2}{1} \sm S_{R_2}|  \geq \deg_{G-E(R_2)}(w,\I{2}{1}) \geq (2/3 + 8 \gamma) n - |V(R_{1})|.
	$$
	Summing the two inequalities (recalling that $S_{R_2} \subset \I{2}{1}$) gives
	\[
		|\I{2}{1}| > (4/3 + 15 \gamma) n - |V(R_1\cup R_2)| = (1/3 + 15 \gamma)n,
	\]
	in contradiction to~\eqref{equ:I<1/3}.
\end{claimproof}
\begin{claim}\label{cla:SRi-in-unique-Ijj'}
	For each $i,j \in [2]$, the set $\I{i}{j}$ contains at most one of the sets $S_{R_1}, S_{R_2}, S_{B_1}, S_{B_{2}}$.
\end{claim}
\begin{claimproof}
	Suppose the claim is wrong for $i=j=2$, namely, we suppose that
	$S_{R_2} ,S_{B_2} \subset \I{2}{2}.$
	Then by Claim~\ref{cla:SR2capB1-neq-es->I11-subs-NSRS-cap-B1},
	\begin{equation}\label{*****+++++}
		\I{2}{1}\subseteq N_{G-E(B_2)}(S_{B_2})\text{ and } \I{1}{2}\subseteq N_{G-E(R_2)}(S_{R_2}).
	\end{equation}
	As $S_{R_2}$ and $S_{B_2}$ are each contracting, this gives
	\begin{equation*}
		|\I{2}{1}|  < |S_{B_2}| + \gamma n\text{ and } |\I{1}{2}|  < |S_{R_2}| + \gamma n.
	\end{equation*}
	So, for any $v \in \I{2}{2}$, we have
	\begin{align*}
		\deg_G(v,\I{2}{2}) & \geq (2/3 + 8 \gamma) n-|\I{1}{2}|-|\I{2}{1}|                    \\
		                   & \geq |\I{2}{2}| -|\I{1}{2}| + |\I{2}{2}|-|\I{2}{1}| + 24\gamma n \\
		                   & > |\I{2}{2}| -|S_{R_2}| + |\I{2}{2}|-|S_{B_2}| + 22\gamma n
	\end{align*}
	where the second inequality follows from~\eqref{equ:I<1/3}. Hence $v$ satisfies at least one of the following:
	\[\deg_\b(v,\I{2}{2}) > |\I{2}{2}|-|S_{R_2}| + 11\gamma n
		\text{ \ \ or \ \ }
		\deg_\r(v,\I{2}{2}) > |\I{2}{2}|-|S_{B_2}| +11\gamma n.\]
	Let us call $T_B$ the set of all vertices in $\I{2}{2}$ satisfying the first inequality and let $T_R$ be the set of all vertices in $\I{2}{2}$ satisfying the second inequality.
	Note that $T_R\cup T_B=\I{2}{2}$.
	Moreover, vertices in $T_B$ must each have a blue edge to $S_{R_2}$, and similarly vertices in $T_R$ have each a red edge to~$S_{B_2}$.
	Putting this together with~\eqref{*****+++++}, we obtain
	$$T_R\cup \I{2}{1} \subseteq N_{G-E(B_2)}(S_{B_2})\text{ and }T_B\cup \I{1}{2} \subseteq N_{G-E(R_2)}(S_{R_2}).$$

	By definition of $T_R$ and since the set $S_{R_2}$ is stable in $G-E(R_2)$, we have $S_{R_2}\subseteq T_R$. So, because of~\eqref{equ:dense-Iij-min-deg}, $$|S_{R_2}|\le | T_R|\le |N_{G-E(B_2)}(S_{B_2})| - |\I{2}{1}|\le |S_{B_2}| +\gamma n- 24\gamma n< |S_{B_2}|,$$ and similarly, $|S_{B_2}|<  |S_{R_2}|$, a contradiction.
\end{claimproof}

By Claim~\ref{cla:SRi-subset-Ijj'} and Claim~\ref{cla:SRi-in-unique-Ijj'}
we can assume without loss of generality that $S_{R_2} \subset \I{2}{1}$, $S_{B_1} \subset \I{1}{1}$, $S_{R_1} \subset \I{1}{2}$ and $S_{B_2} \subset \I{2}{2}$.
Thus Claim~\ref{cla:SR2capB1-neq-es->I11-subs-NSRS-cap-B1} yields the following inequalities
\begin{align*}
	|\I{1}{1}| & \leq |N_{G-E(R_2)}(S_{R_2})| < |S_{R_2}| + \gamma n\leq |\I{2}{1}| + \gamma n;           \\
	|\I{2}{1}| & \leq |N_{G-E(B_2)}(S_{B_2})| < |S_{B_2}| + \gamma n\leq |\I{2}{2}|+ \gamma n;            \\
	|\I{2}{2}| & \leq|N_{G-E(R_1)}(S_{R_1})| < |S_{R_1}|+ \gamma n \leq |\I{1}{2}|+ \gamma n; \text{ and} \\
	|\I{1}{2}| & \leq|N_{G-E(B_1)}(S_{B_1})| < |S_{B_1}| + \gamma n\leq |\I{1}{1}|+ \gamma n.
\end{align*}
Hence the colouring is $(4\gamma)$-extremal as in  Definition~\ref{def:extremal-colouring}\ref{conf:four-cycle}, which contradicts our initial assumption.
This finishes the proof of Lemma~\ref{lem:components}.


\section{Covering exceptional vertices: The proof of Lemma~\ref{lem:distribute-exceptional-vertices}}
\label{sec:distribute-exceptional-vertices}

The proof of Lemma~\ref{lem:distribute-exceptional-vertices} goes roughly as follows.
We first pick two sparse and pairwise disjoint families of paths $\Pedge$ and $\Pexc$ such that $\Pedge$ supports paths embeddings in the $H_i$'s and~$\Pexc$ contains the vertices $W$.
We choose the families such that each of their paths connects to one of the $H_i$'s.
This allows Lemma~\ref{lem:connecting-path} to connect these families up to three monochromatic cycles $C = C_1 \cup C_2 \cup C_3$.
Finally, we  adjust the parity of $V(G) \sm V(C)$ using either an odd cycle in one of the $H_i$'s or the fact that otherwise $H_3 =\es$, in which case we can set $C_3$ to be a single vertex.

Before we start, let us formalize what we mean by  a path connecting to one of the $H_i$'s.
\begin{definition}[$\delta$-connect]
	Suppose $G$ is a $2$-edge-coloured graph on $n$ vertices and let $\V = \{V_1,\ldots,V_m\}$ be an $(\eps,d)$-regular partition of $G$.
	Let $\RdG$ be the corresponding  $2$-edge-coloured $(\eps,d)$-reduced multi-graph.
	For $\delta > 0$, we say that a path $P \subset G_\chi$ $\delta$-connects to $H \subset \RdG_\chi$, if both ends of $P$ have at least $\delta|V_x|$ edges of colour $\chi$ leading to the same cluster $V_x$ for some $x \in V(H)$.
\end{definition}

\begin{proof}[Proof of Lemma~\ref{lem:distribute-exceptional-vertices}]
	Let ${1}/{n}  \ll \eps, 1/m \ll d \ll \beta \ll 1$ such that Lemma~\ref{lem:connecting-path} holds with input $\eps, d, m, \ell=m+2, n$.
	Let $t$ be the size of the clusters $V_1,\ldots,V_m$.
	\begin{claim}\label{cla:short-edge-paths}
		Let $S \subset V(G)$ be $\delta$-sparse in $\V$.
		Then there is a family $\Pedge$ of at most $2m^2$ monochromatic paths  with the following properties:
		\begin{itemize}
			\item the paths of $\Pedge$ are pairwise disjoint and disjoint from $S$;
			\item For each $i$, the paths of $\Pedge$ with the same colour as $H_i$ support path embeddings in $H_i$;
			\item each path of $\Pedge$ $(d/2)$-connects to one of the $H_i$'s.
		\end{itemize}
	\end{claim}
	\begin{claimproof}
		We choose the paths $\Pedge$ greedily as follows. For each $i\in[3]$, we order arbitrarily the edges of $H_i$. When we come to the edge $pq$ of $H_i$, we choose a $2$-edge path with both ends in $V_p$ and the middle vertex in $V_q$, which is disjoint from previously chosen paths and from $S$, whose end-vertices both have at least $dt/2$ neighbours in the colour of $H_i$ in $V_q$. Since at most $2\eps t$ vertices of~$V_p$ do not have degree at least~$d t/2$ to~$V_q$ in the colour of~$H_i$, and at most~$6m^2+\delta|V_p|$ vertices of~$V_p$ are contained in previously chosen paths or in~$S$, we are left with at least~$\tfrac12|V_p|$ vertices of~$V_p$ to choose the endpoints from, and similarly at least~$\tfrac12|V_q|$ vertices of~$V_q$ to choose the middle vertex from. Since $V_pV_q$ is an edge of $H_i$, by definition of $\eps$-regularity there is such a $2$-edge path in the colour of $H_i$ as desired.
	\end{claimproof}
	The next two claims take care of the exceptional vertices.
	\begin{claim}\label{cla:distribute-exc-vertices-a}
		Suppose that $|\bigcup_{i\neq j} V(H_i) \cap V(H_j)| \geq m/3$.
		Let $S$ be $\delta$-sparse in $\V$.
		Then there is a family $\Pexc$ of at most $\beta^{-2}$ monochromatic paths with the following properties:
		\begin{itemize}
			\item the paths of $\Pexc$ are pairwise disjoint and disjoint from $S$;
			\item the set $V(\Pexc)$ contains $W$ and is $\sqrt{\delta}$-sparse in $\V$;
			\item each path of $\Pexc$ $(d/2)$-connects to one of the $H_i$'s.
		\end{itemize}
	\end{claim}
	\begin{claimproof}
		Let us define
		\[\Z= \bigcup_{1 \leq i\neq j \leq 3} V(H_i) \cap V(H_j) \subset V(\RdG) ,\quad  Z = \bigcup_{x \in \Z} V_x\subset V(G). \]
		We set $r =\lceil 4/\beta\rceil$.
		By assumption we have $|\Z|\ge\tfrac{m}{3}$ and hence $$|Z| \geq (n-|V_0|)|\Z|/m\geq n/3 - \eps n.$$
		So as $\delta(G) \geq (2/3 + \beta)n$, it follows that $\deg_G(w,Z) \geq \beta n/2 \geq 2n/r$ for any $w \in V(G)$.
		We partition $W$ by setting $W_1 := \{v \in W\colon \deg_\r(v,Z) \geq n/r\}$ and $W_2 := W \sm W_1$.
		For $i \in [2]$ define an auxiliary graph $A_i$ on vertex set $W_i$, by connecting two vertices $v,w \in W_i$ if $|N_i(v,Z) \cap N_i(w,Z)| \geq n/r^3$.
		We claim that
		\begin{equation}\label{equ:independence-number}
			\text{the independence number of $A_i$ is bounded by $r$.}
		\end{equation}
		Indeed, suppose otherwise and let $w_1, \ldots, w_{r+1}$ be pairwise disjoint non-adjacent vertices in $A_i$.
		We then obtain the following contradiction.
		\begin{align*}
			|Z| & \geq \left|\bigcup_{p\in[r+1]} N_{i}(w_p,Z) \right|                                   \\
			    & \geq (r+1) \frac{n}{r} - \sum_{1 \leq q <p \leq r+1} |N_{i}(w_q,Z) \cap N_{i}(w_p,Z)| \\
			    & \geq |Z| \Big( \frac{r+1}{r} - \frac{\binom{r+1}{2}}{r^3} \Big) > |Z|.
		\end{align*}
		This proves~\eqref{equ:independence-number}.
		Now a classic result of P\'osa (see~\cite{Lov79}) guarantees that $A_i$ can be partitioned into $r_i\leq r$ disjoint cycles $C_{i,1}, \ldots, C_{i,r_i}$, where we consider edges and vertices to be cycles of lengths respectively $2$ and $1$.

		To finish, we build from each $C_{i,j}$ a path $P_{i,j}$ of colour $i$  that $(d/2)$-connects to one of $H_1,H_2,H_3$, such that the union of these paths is $\sqrt{\delta}$-sparse in $\V$.
		In addition to this, we require the paths $P_{i,j}$ to be pairwise disjoint and disjoint from $S$.
		This can be done greedily as follows. We break $C_{i,j}$ into a path by removing one edge, and then create $P_{i,j}$ by connecting one after another two adjacent vertices of this path with a vertex in $Z$.

		To see that we never run out of available vertices in this process, note that by definition of $r$ each two adjacent vertices in one of the $C_{i,j}$'s share $(\beta/4)^3n$ common neighbours of colour $i$ in $Z$.
		So each two adjacent vertices in one of the $C_{i,j}$'s have at least $(\beta^3/8)t \geq (d/2)t$ common neighbours in some cluster of $\Z$.
		Moreover, the restriction of $\sqrt{\delta}$-sparseness and disjointness renders at most
		$$\frac{|W|}{\sqrt{\delta}t} \cdot t+ 2|W| + |S| \leq 3\sqrt{\delta}n$$
		vertices unavailable at any point of embedding.

		Finally, note that since there is an edge in $C_{i,j}$ between the first and last vertices $w$ and $z$ of $P_{i,j}$, so $w$ and $z$ have at least $n/r^3$ common neighbours in $Z$ in colour $i$. By averaging, some cluster~$V_p$ of~$\Z$ contains at least~$t/r^3$ of these common neighbours. Picking $H_{i,j}$ to be one of $H_1,H_2,H_3$ which contains $V_p$ and is of colour $i$ (which must exist since $V_p$ is in two of the $H_1,H_2,H_3$ and these two have different colours by condition~\ref{itm:distr-exc-different-colours} of Lemma~\ref{lem:distribute-exceptional-vertices}) we see that $P_{i,j}$ is $(d/2)$-connected to $H_{i,j}$.
	\end{claimproof}
	\begin{claim}\label{cla:distribute-exc-vertices-b}
		Suppose that~\ref{itm:distr-exc-i}\ref{itm:distr-exc-spanning-and-empty} of the assumptions of Lemma~\ref{lem:distribute-exceptional-vertices} holds.
		Let $S$ be $\delta$-sparse in $\V$.
		Then there is a family $\Pexc$ of at most $\beta^{-2}$ monochromatic paths and a monochromatic cycle $C_3$
		with the following properties:
		\begin{itemize}
			\item the paths of $\Pexc$ are pairwise disjoint and disjoint from $S$;
			\item the set $V(\Pexc) \cup V(C_2)$ contains $W$ and is $\sqrt{\delta}$-sparse in $\V$;
			\item each path of $\Pexc$ $(d/2)$-connects to $H_1$.
		\end{itemize}
	\end{claim}
	\begin{claimproof}
		The proof is very similar to the one of Claim~\ref{cla:distribute-exc-vertices-a}.
		Without loss of generality we can assume that $H_1$ is red.
		We define
		\[\Z= V(H_1) = V(\RdG) ,\quad  Z = \bigcup_{x \in \Z} V_x\subset V(G). \]
		By the assumptions we have $$|Z| \geq (n-|V_0|)|\Z|/t \geq (1-\eps)n.$$
		It follows that $\deg_G(w,Z) \geq (2/3 + \beta /2)$ for any $w \in V(G)$.
		As before, we us set $r =\lceil 4/\beta\rceil$.
		We can partition $W$ by setting $W_1 := \{v \in W\colon\deg_\r(v,Z) \geq n/r\}$ and $W_2 := W \sm W_1$.
		For the set~$W_1$ we proceed as in Claim~\ref{cla:distribute-exc-vertices-a}.
		For~$W_2$, however, note we now have $\deg_2(v,Z) \geq (2/3 + \beta /4)n$ for every $v \in W_2$.
		Hence we can greedily construct a single blue cycle $C_3$ for the vertices $W_2$, which has the desired properties.
	\end{claimproof}

	Now we can finish the proof of Lemma~\ref{lem:distribute-exceptional-vertices}.
	If $H_i$ has no edges for some $i \in [3]$ we replace it with $H_i = \es$. Since $H_1$ and $H_3$ are monochromatic components, and $H_2$ is the union of at most two monochromatic components, if $H_i$ has no edges then it has at most two vertices. So after this step the assumptions of the lemma still hold with 	\ref{itm:distr-exc-i}\ref{itm:distr-exc-intersection} relaxed to $|\bigcup_{i\neq j} V(H_i) \cap V(H_j)| \geq (1/3+\gamma/2)m$.
	Next, if $H_i = \es$ for some $i \in [2]$, we take $C_i$ to be the empty set.
	If $H_3 = \es$ and we are not in case \ref{itm:distr-exc-i}\ref{itm:distr-exc-spanning-and-empty}, we also take $C_3$ to be the empty set.
	(Note that in case \ref{itm:distr-exc-i}\ref{itm:distr-exc-spanning-and-empty}, $C_3$ plays a more active role as detailed in Claim~\ref{cla:distribute-exc-vertices-b}.)
	In all other cases we construct the $C_i$'s as follows.

	Let us first handle the case, where~\ref{itm:distr-exc-i}\ref{itm:distr-exc-intersection} of the assumptions of Lemma~\ref{lem:distribute-exceptional-vertices} holds.
	By assumption~\ref{itm:distr-exc-H1H2H3}, if $H_2\neq\es$, then $H_2$ is the union of monochromatic components $F_1$ and $F_2$, which admit bridges (Definition~\ref{def:bridges}).
	Hence there are distinct vertices $u,u'$ in $G$, which each $d$-connect to both~$F_1$ and $F_2$, and we let  $\Pbri = \{u,u'\}$ be the set of these two one-vertex paths.  If $H_2=\es$ then we set $\Pbri=\es$.
	We fix a family of paths $\Pedge$ obtained from an application of Claim~\ref{cla:short-edge-paths} with $S = V(\Pbri)$.
	Next we fix a family of paths $\Pexc$ of size at most $\beta^{-2}$ obtained from an application of Claim~\ref{cla:distribute-exc-vertices-a} with $S = V(\Pedge \cup \Pbri)$.
	Note that $\Pbri \cup \Pedge \cup \Pexc$ is $2\sqrt{\delta}$-sparse in $\V$ and contains at most $2+2t^4+ \beta^{-2}$ paths.

	We partition the paths $\Pedge \cup \Pexc = \cP_1\cup \cP_2 \cup \cP_3$ such that each path in  $\cP_i$ is $(d/2)$-connected to $H_i$.
	For $H_2$ we partition $\cP_2 = \cP_{2,1} \cup \cP_{2,2}$ further such that each path in $\cP_{2,j}$ is $(d/2)$-connected to $F_j$.
	To finish we apply Lemma~\ref{lem:connecting-path} to connect the $\cP_i$'s up to three monochromatic cycles.
	At this point it is important that $H_1,H_3,F_1,F_3$ are monochromatic connected subgraphs.

	For $H_1$ we orient the paths of $\cP_1$ arbitrarily and then choose an arbitrary cyclic order of $\cP_1$.
	We then use Lemma~\ref{lem:connecting-path} to connect one of the $dm$ neighbours of the end of each path of $\cP_1$ with one of the $dm$ neighbours of the start of the next path of $\cP_1$ (in the chosen order). Note that since $H_1$ is a connected graph on at most $m$ vertices, any such connection requires a path with at most $m+1$ vertices.
	This way we obtain the cycle $C_1$ and similarly $C_3$.
	For $H_2$ we want to use the paths $\Pbri$.
	So for each $j\in[2]$ we orient the paths of $\cP_{2,j}$ arbitrarily, then choose an order on the paths of $\cP_{2,j}$ which starts in a cluster with many colour-$H_2$-neighbours of $u$.
	We then proceed as before to obtain the cycle $C_2$, except that after connecting the paths of $\cP_{2,j}$ we add one further connecting path to a colour-$H_2$-neighbour of $u'$.

	Recall that $\Pbri \cup \Pedge \cup \Pexc$ is $2\sqrt{\delta}$-sparse in $\V$ and contains at most $2+2t^4+ \beta^{-2}$ paths.
	Since each of the application of Lemma~\ref{lem:connecting-path} yields a path of order at most $2t^2$ and $(2+2t^4+ \beta^{-2})2t^2 \ll \sqrt{\delta} m$, it follows that $C = C_1 \cup C_2 \cup C_3$ is $(3\sqrt{\delta})$-sparse in $\V$.
	Hence by setting $S$ to contain the vertices of all paths chosen so far in each of the above applications of Lemma~\ref{lem:connecting-path}, we can guarantee that the $C_i$'s are vertex disjoint.

	It remains to ensure that $V(G) \sm C$ has even size.
	So suppose this is not the case.
	If each of the~$H_i$'s is bipartite, then $H_3 = \es$ by assumption and so $C_3 = \es$.
	In this case we set $C_3 = v$ for some unused vertex $ v \in V(G) \sm V(C)$.
	On the other hand, suppose, $F_1$ say, has an odd cycle.
	Note that this implies that $F_1$ contains walks of any parity between any two vertices.
	Let $P$ be one of the paths obtained by Lemma~\ref{lem:connecting-path} and which connects two paths of $\cP_{2,1}$.
	We use Lemma~\ref{lem:connecting-path} to swap~$P$ for a path $P'$ whose order has the opposite parity to that of $P$.
	Then $C' = C_1 \cup C_2 \cup C_3 \sm P \cup P'$ has the desired parity.
	The cases where $H_1$, $H_3$, or $F_2$ have an odd cycle are handled the same way.
	This finishes the proof for the case~\ref{itm:distr-exc-intersection}.

	Now assume that~\ref{itm:distr-exc-i}\ref{itm:distr-exc-spanning-and-empty} of the assumptions of Lemma~\ref{lem:distribute-exceptional-vertices} holds.
	We proceed very similarly.
	The main difference is that the paths returned by Lemma~\ref{cla:distribute-exc-vertices-b} correspond only to $H_1$ and $H_2$, i.e. $\cP_3 = \es$.
	So for $H_1,H_2$ we can proceed as in the case of~\ref{itm:distr-exc-intersection} and obtain cycles $C_1,C_2$.
	Finally, we take $C_3$ to be the cycle returned by Lemma~\ref{cla:distribute-exc-vertices-b}, which is possible as by~\ref{itm:distr-exc-spanning-and-empty} we have $H_3 = \es$.
	It remains to ensure that $V(G) \sm C$ has even size.
	As in case~\ref{itm:distr-exc-spanning-and-empty} we are guaranteed that $H_1$ or $H_2$ contains an odd cycle, we may swap a path $P$ for a path $P'$ whose order is of the opposite parity as we did in  case~\ref{itm:distr-exc-intersection} to obtain a set of cycles of the desired parity.
\end{proof}


\section{Balancing: The proof of Lemma~\ref{lem:balancing}}\label{sec:balancing}

In this section we prove Lemma~\ref{lem:balancing}. The idea is simple. First we
set some $t_{e}$ equal to one and others to zero such that the shortfall at each
vertex is even. Then we create an auxiliary graph $H'$ obtained by blowing up
each vertex of $H$ to exactly half its shortfall number of vertices, and find a
$2$-matching in $H'$; this tells us how much to increment each $t_e$ by.

\begin{proof}[Proof of Lemma~\ref{lem:balancing}]
	Suppose $t\ge m/\gamma$ and let~$H$ be a
	connected $m$-vertex graph that is $\gamma$-robust Tutte.
	For a vertex~$i$ in~$H$ denote by~$E_i$ the set of edges in~$H$ such that $i\in e$.

	We start with adjusting parity for~$i$ by defining $\pi_e$ for $e\in E(H)$ such that
	$t_i-\sum_{e\in E_i}\pi_e$ is even. For this purpose
	let $T$ be an arbitrary spanning tree in $H$ with root $r$. The idea is that
	for the leaves~$i$ of this tree, the value $\pi_e$ for the single edge~$e$ containing~$i$ will
	adjust the parity of~$t_i$. Then we will work upwards from the leaves to
	the root, iteratively adjusting the parity of~$t_i$ for a vertex~$i$ by
	appropriately setting $\pi_e$ for the unique edge~$e$ containing~$i$ in the
	path from~$i$ to~$r$. How we set this $\pi_e$ depends on the value of the
	$\pi_{e'}$ for the other edges containing~$i$ that we set previously.

	Formally, for each $ij=e\in E(T)$ with~$j$ closer to the root than~$i$, let
	$X(e)$ denote the set of vertices of the component of $T-e$ that contains~$i$.
	If $e\not\in E(T)$ then we set $\pi_e=0$. For $e\in E(T)$ let
	$\pi_e=(\sum_{j\in X(e)}t_j)\!\!\mod 2$. We have for each $i\in V(H)$ that
	\begin{equation}\label{eq:parity}
		t_i-\sum_{e\in E_i}\pi_e
		=t_i-\sum_{e\in E(T)\cap E_i}\pi_e
		=t_i-\sum_{e\in E(T)\cap E_i}\bigg(\Big(\sum_{j\in X(e)}t_j\Big)\hspace{-.8em}\mod 2\bigg)\,.
	\end{equation}
	Now, let $e_1,\dots,e_s$ be the edges in~$E_i$, where~$e_1$
	is the unique edge among these such that~$j$ is closer to the root than~$i$.
	Observe that we have $X(e_\ell)\cap X(e_{\ell'})=\es$ for each
	$2\le\ell<\ell'\le s$ and
	$X(e_1)=\{i\}\cup \bigcup_{2\le\ell\le s} X(e_\ell)$.
	This implies that in the last sum in~\eqref{eq:parity} the terms appearing are
	exactly the following: The term $t_i$ appears once, and for each $j\in
		X(e_1)\setminus\{i\}$ the term $t_j$ appears twice. We conclude that
	the quantity $2n_i:=t_i-\sum_{e\in E_i}\pi_e$
	is an even integer, so that $n_i$ is an integer.
	Since $t\ge 5m/\gamma$ we have $|E_i|\le m\le\gamma t/5$ for each $i$, and so because
	$t_i=(1\pm\frac15\gamma)t$ we get $n_i=(1\pm\frac25\gamma)t/2$.

	In the second part of this proof, we will define non-negative integers $t''_e$ such that
	\begin{equation}\label{eq:shortfall}
		\sum_{e\in E_i} t''_e = 2n_i\,.
	\end{equation}
	This suffices to prove the lemma since setting
	$t_e:=t''_e+\pi_e$ gives non-negative values such that
	$\sum_{e\in E_i} t_e = 2n_i+\sum_{e\in E_i}\pi_e = t_i$ as desired.

	For obtaining these~$t''_e$, consider the graph $H'$ obtained from $H$ by
	blowing up each $i$ into an stable set~$W_i$ of $n_i$ vertices, and
	replacing edges with complete bipartite graphs. Let $S'$ be a stable set in
	$H'$. We aim to show $\big|N_{H'}(S')\big|\ge|S'|$. Without loss of generality,
	we can assume $S'$ is a union of the blow-ups of some vertices $S\in V(H)$
	(otherwise we can add any vertices missing from such a blow-up to $S'$ without
	changing the left hand side), where $S$ must be a stable set in $H$. Then,
	since~$H$ is $\gamma$-robust Tutte, we have
	\begin{equation*}\begin{split}
			\big|N_{H'}(S')\big|& = \sum_{i\in N_H(S)} n_i
			\ge(1-\tfrac25\gamma)\tfrac{t}{2}\big|N_H(S)\big| \\
			&\ge(1-\tfrac25\gamma)(1+\gamma)\tfrac{t}{2}|S|
			\ge \frac{1-\tfrac25\gamma}{1+\tfrac25\gamma}(1+\gamma)|S'|\ge|S'|\,,
		\end{split}\end{equation*}
	where the last inequality uses $\gamma\le\frac12$. Thus by
	Theorem~\ref{thm:tutte-2-matching} there is a perfect $2$-matching $\Mblow$ in $H'$.
	For $e=ij\in E(H)$, recall that $W_i\times W_j$ is the blow-up of this edge
	in~$H'$. Now set $t''_e:=\sum_{f\in W_i\times W_j}\omega(f)$, where $\omega(f)=2$ if $e$ is on an edge of $\Mblow$, $\omega(f)=1$ if $e$ is on a cycle of $\Mblow$, and $\omega(e)=0$ otherwise.
	We have
	\begin{equation*}
		\sum_{e\in E_i} t''_e =
		\sum_{ij\in E_i}\sum_{f\in W_i\times W_j}\omega(f) =
		\sum_{w\in W_i}\sum_{f\in E(H')\colon w\in f}\omega(f) =
		n_i\cdot 2\,,
	\end{equation*}
	where the last equality uses that $|W_i|=n_i$, the definition of~$\omega$ and the fact that $\Mblow$ is
	a perfect $2$-matching.
	Thus we obtain~\eqref{eq:shortfall}, as required.
\end{proof}


\section{Extremal cases: The proof of Lemma~\ref{lem:extremal-colouring}}
\label{sec:extremal-colourings}

For the proof of Lemma~\ref{lem:extremal-colouring} it is natural to distinguish
two cases corresponding to the two cases in the definition of extremal colourings.
For both cases we choose
\begin{equation*}
	{1}/{n}  \ll \eps,1/m \ll d \ll \gamma \ll \beta\,,
\end{equation*}
suitable for various estimates and we
suppose $G$ is a graph on $n$ vertices with minimum degree $\delta(G) \geq (2/3 +
	\beta)n$ whose edges are coloured with red and blue, and with
a coloured $(\eps,d)$-reduced multi-graph $\RdG$ on $m$ vertices corresponding
to an $(\eps,d)$-regular partition $\V=\{V_1,\dots,V_m\}$ of~$G$ with
exceptional set~$V_0$. By~\eqref{equ:reduced-graph-min-deg} we have
\begin{equation*}
	\delta(\RdG)\ge(2/3 + \beta/2) m\,.
\end{equation*}
Further, we let~$G'$ be the ``cleaned'' coloured subgraph of~$G$ on vertex set
${V(G)\setminus V_0}$ corresponding
to the $(\eps,d)$-reduced graph~$\RdG$ with properties as guaranteed by
Lemma~\ref{lem:regularity}: In particular, each red (blue) edge of~$G'$ is in
some pair $(V_i,V_j)$ such that $ij$ is a red (blue) edge of~$\RdG$.
Further, let $G''$ be the spanning subgraph of~$G$ which coincides with~$G'$ on
$V(G)\setminus V_0$ and has all $G$-edges incident to vertices in~$V_0$.
Clearly,
\begin{equation*}
	\delta(G'),\delta(G'')\ge(2/3+\beta-2d-\eps)n\ge(2/3+\beta/2)n\,.
\end{equation*}

\subsection{Case~1: A colouring as in Definition~\ref{def:extremal-colouring}\ref{conf:bipartite}}

We start with the case that the coloured reduced multi-graph has one spanning
bipartite connected red component, and exactly two blue components and these are
bipartite as well and do not admit bridges.

\begin{proof}[Proof of Lemma~\ref{lem:extremal-colouring} when $\RdG$ has a
		colouring as in Definition~\ref{def:extremal-colouring}\ref{conf:bipartite}]
	Let $R$ be the spanning, red say, bipartite component of $\RdG$ with bipartition classes $\cX_{R}$ and $\cY_{R}$.
	As the colouring is $(4\gamma)$-extremal we have $\left||\cX_{R}| - |\cY_{R}| \right| \leq 4 \gamma m$ and hence
	\begin{equation}\label{equ:|X|=|Y|}
		m/2- 4\gamma m \leq |\cX_{R}|,|\cY_{R}| \leq m/2+ 4\gamma m.
	\end{equation}
	Set $X_{{R}} := \bigcup_{i \in \cX_{{R}}} V_i$ and $Y_{{R}} :=
		\bigcup_{i \in \cY_{{R}}} V_i$. In particular, any vertex~$x$
	in~$X_{R}$ (in~$Y_{R}$, respectively) sends red edges in~$G'$ only to
	vertices in~$Y_{R}$ (in~$X_{R}$, respectively) and blue edges in~$G'$ only
	to vertices in~$X_{R}$ (in~$Y_{R}$, respectively).
	Let $B_j$ with $j=1,2$ be the bipartite blue components
	of~$\RdG$ with bipartition classes $\cX_{B_j},\cY_{B_j}$. Suppose that $V(B_1)
		= \cX_{R}$ and $V(B_2) = \cY_{R}$. Let
	$X_{{B_j}} := \bigcup_{i \in \cX_{{B_j}}} V_i$ and $Y_{{B_j}} :=
		\bigcup_{i \in \cY_{{B_j}}} V_i$ for $j=1,2$. Thus $\{X_{B_1},Y_{B_1}\}$ is a partition of $X_{R}$ and $\{X_{B_2},Y_{B_2}\}$ is a partition of $Y_{R}$.

	Let us next estimate the sizes of the $\cX_{B_j}$, $\cY_{B_j}$.
	By~\eqref{equ:|X|=|Y|}, we have \[\deg_{G'_{\b}}(x,\cX_{R}), \deg_{G'_\b}(y,\cX_{R})
		\geq (1/6 + 4 \gamma) m\] for any $x,y \in \cX_{R} $. If $xy$ is additionally
	a (blue) edge, then bipartiteness implies $N_{G'_\b}(x,\cX_{R}) \cap
		N_{G'_\b}(y,\cX_{R}) = \es$ and hence
	\begin{equation}
		|\cX_{B_1}|,|\cY_{B_1}| \leq |\cX_{R}| - (1/6 + 4 \gamma)m \overset{\eqref{equ:|X|=|Y|}}{\leq} m/3.
	\end{equation}
	By symmetry we also have $|\cX_{B_2}|,|\cY_{B_2}| \leq  m/3$.
	We conclude for $j=1,2$ that
	\begin{equation}
		\label{equ:B1,B2>1/3}
		|X_{B_j}|,|Y_{B_j}|\le n/3\,.
	\end{equation}

	We next distribute the exceptional vertices in~$V_0$ as follows.
	We let
	\begin{align*}
		X_{{R}}' & = \{v \in V(G)\colon\deg_{G''_\b}(v,Y_{{R}}) \le \tfrac1{10}\beta n \}, \\
		Y_{{R}}' & = \{v \in V(G)\colon\deg_{G''_\b}(v,X_{{R}}) \le \tfrac1{10}\beta n\}.
	\end{align*}
	Clearly, $X_{R}\subset X_{R}'$ since vertices in~$X_{R}$ send no
	blue edges to $Y_{R}$ in~$G'$ (and hence in~$G''$); analogously,
	$Y_{R}\subset Y_{R}'$. By~\eqref{equ:|X|=|Y|} this implies
	\begin{equation}\label{equ:|X'|=|Y'|}
		n/2- 5\gamma n \leq |X'_{R}|,|Y'_{R}| \leq n/2+ 5\gamma n\,.
	\end{equation}
	Moreover, since by assumption $(B_1,B_2)$ does not admit bridges, $V(G) \sm (X_{{R}}'
		\cup Y_{{R}}')$ contains at most one vertex; call this vertex~$z^*$ if it
	exists, and otherwise we say $z^*$ does not exist.

	Without loss of generality suppose that $|X_{{R}}'| \geq |Y_{{R}}'|$.
	Our goal is to use two blue cycles to cover $|X_{R}'|-|Y_{R}'|$ vertices
	in~$X_{R}'$ and the vertex~$z^*$, and to use a red cycle to cover all
	remaining vertices. We now prepare for this be establishing some further
	degree bounds.
	By definition, any $x \in X'_{R}$ and $y\in Y'_{R}$ satisfy
	\begin{equation}\label{equ:deg-in-Y-large}\begin{split}
			\deg_{G''_\r}(x,Y'_{R}) & \geq (\tfrac23 + \tfrac12\beta)n - |X'_{R}| - \tfrac1{10}\beta n - 1 \\
			& \geq (\tfrac23 + \tfrac12\beta)n - (\tfrac12 + 5\gamma)n - \tfrac1{10}\beta n - 1
			\geq (\tfrac16 + \tfrac14\beta)n\,, \\
			\deg_{G''_\r}(y,X'_{R}) &\ge (\tfrac16 + \tfrac14\beta)n\,.
		\end{split}\end{equation}
	For $x\in X_{R}$ and $y\in Y_{R}$ we know even more: The vertex~$x$ is
	either in~$X_{B_1}$ or in~$Y_{B_1}$. In the former (latter) case it
	sends all its $G''$-edges in~$X_{R}$ to~$Y_{B_1}$
	(to~$X_{B_1}$). This together with~\eqref{equ:B1,B2>1/3} (and an analogous argument
	for~$y$) gives
	\begin{equation}\label{equ:deg-in-Y-larger}\begin{split}
			\deg_{G''_\r}(x,Y'_{R}) & \geq (\tfrac23 + \tfrac12\beta)n-|Y_{B_1}| - |V_0| - \tfrac1{10}\beta n - 1 \\
			& \ge (\tfrac13 + \tfrac14\beta)n\,. \\
			\deg_{G''_\r}(y,X'_{R})
			& \ge (\tfrac13 + \tfrac14\beta)n\,.
		\end{split}\end{equation}
	Concerning blue edges incident to $x\in X_{R}$ and $y\in Y_{R}$ we know
	\begin{equation}\label{equ:blue-deg-in-X}
		\deg_{G'_\b}(x,X_{R}) \ge (\tfrac23 + \tfrac12\beta)n-|Y_{R}|-|V_0|
		\ge (\tfrac16 + \tfrac14\beta)n\,.
	\end{equation}

	Let us now construct our three cycles.
	We first construct a blue cycle~$C_2$ with $V(C_2) \subset X_{R}$, of length
	\[2\ell=2\lceil (|X_{{R}}'| - |Y_{{R}}'|)/2\rceil \leq 10\gamma n + 1 \leq 11 \gamma n\,.\]
	For this, observe that the blue bipartite graph $F=G'_\b[X_{B_1},Y_{B_1}]$ has
	partition classes of size at most~$n/3$ by~\eqref{equ:B1,B2>1/3} and minimum
	degree at least $(\frac16+\frac14\beta)n$ by~\eqref{equ:blue-deg-in-X}.
	Hence, picking an arbitrary vertex $x\in X_{B_1}$ we can greedily construct
	two paths $P_1$, $P_2$ with $\ell-1$ edges starting in~$x$ which are
	otherwise vertex-disjoint. Let~$y_1$ be the endpoint of~$P_1$ and~$y_2$ the
	endpoint of~$P_2$. Clearly, $y_1$ and~$y_2$ are in the same partition class
	of~$F$ and thus have a common neighbour~$x'$ in $V(F)\setminus\big(V(P_1)\cup V(P_2)\big)$
	since $\gamma\ll\beta$. Thus, $P_1$, $P_2$ and~$x'$ give the desired blue
	cycle~$C_2$ on $2\ell$ vertices.

	Our second cycle is also blue, and takes care of parity and the vertex~$z^*$ if it exists.
	If $|X_{{R}}' \sm V(C_2)| = |Y_{{R}}' \sm V(C_2)|$, we set $C_3:= z^*$ (a one-vertex cycle) if $z^*$ exists, and if $z^*$ does not exist we let $C_3$ be the zero-vertex cycle. If
	$|X_{{R}}' \sm V(C_2)| = |Y_{{R}}' \sm V(C_2)|+1$ and $z^*$ exists, then we choose $y\in X'_R\sm V(C_2)$ such that $yz^*$ is a blue edge (which exists by definition of $z^*$) and let~$C_3$ be the $2$-vertex blue cycle $(y,z^*)$. Finally if $|X_{{R}}' \sm V(C_2)| = |Y_{{R}}' \sm V(C_2)|+1$ and $z^*$ does not exist, we choose any $y\in X'_R\sm V(C_2)$ and let~$C_3$ be the one-vertex blue cycle $z^*$. Hence we
	have $|X_{{R}}' \sm V(C_2 \cup C_3)| = |Y_{{R}}' \sm V(C_2 \cup C_3)|$.

	To finish we claim that $H=G''_\r[X'_{R} \sm V(C_2 \cup C_3),Y'_{R}\sm V(C_2\cup C_3)]$ has
	a Hamilton cycle~$C_1$. This follows from
	Corollary~\ref{cor:bipartite-chvatal}. Note that the degree conditions are
	fulfilled, because $\delta(H)\geq (1/6+\beta/10)n$
	by~\eqref{equ:deg-in-Y-large}, and all but at most $\eps n$ vertices of~$H$
	have degree at least $(1/3 + \beta/10)n$ by~\eqref{equ:deg-in-Y-larger}.
\end{proof}

\subsection{Case~2: A colouring as in Definition~\ref{def:extremal-colouring}\ref{conf:four-cycle}}

In the second case, the reduced graph has two red and two blue components, and
there are neither bridges in red nor in blue.

\begin{proof}[Proof of Lemma~\ref{lem:extremal-colouring} when $\RdG$ has a colouring as in Definition~\ref{def:extremal-colouring}\ref{conf:four-cycle}]
	We denote the red and blue components of the reduced multi-graph $\RdG$ by $\R_1,\R_2$ and $\B_1,\B_2$ and suppose that there are no red or blue bridges.
	We set $\cI{i}{j} = V(\R_i) \cap V(\B_j)$ for $i,j \in [2]$.
	We may assume that all but at most~$\sqrt{\gamma} m $ vertices $\cQ_{i,j} \subset V(\R_i) \cap V(\B_i)$ have more than $\sqrt{\gamma} m$ red neighbours in $V(\R_i) \cap V(\B_i)$ if $i=j$ and more than $\sqrt{\gamma} m$ blue neighbours otherwise.
	Let $\I{i}{j} = \bigcup_{x \in \cI{i}{j} \sm \cQ_{i,j}} V_x$ for $i,j \in [2]$  and $V_0' = V_0 \cup \bigcup_{x \in \cQ_{i,j}} V_x$.
	We have
	\begin{align}
		(1/4 - {5\sqrt{\gamma}})n & \overset{}{\leq} (1-\eps)|\cI{i}{j}| {n}/{m} \overset{}{\leq} |\I{i}{j}|\overset{}{\leq} |\cI{i}{j}| {n}/{m} \leq ( {1}/{4} + {5\sqrt{\gamma}})n, \label{equ:Iij<1/4} \\
		|V_0'|                    & \leq  |V_0| +  4|\cQ_{i,j}| n/m \leq \eps n + 4\sqrt{\gamma} n \leq 5\sqrt{\gamma} n. \label{equ:size-V_0'}
	\end{align}
	For $i,j \in [2]$ and any vertex~$v\in \I{i}{j}$, all $G'$-neighbours of~$v$ in
	$\I{3-i}{j}$ are blue, all $G'$-neighbours of~$v$ in $\I{i}{3-j}$ are red, $v$
	has no $G'$-neighbours in $\I{3-j}{3-j}$, and all $G'$-neighbours of~$v$ in
	$\I{i}{j}$ are blue if $i=j$ and red otherwise. This implies the following
	neighbourhood profiles of vertices in~$\bigcup_{i,j \in [2]} \I{i}{j}$.

	\begin{claim}\label{cla:I-profiles}
		For each $i,j \in [2]$, every vertex $v \in \I{i}{j}$ we have
		\begin{align*}
			\deg_{G'_\b}(v,\I{3-i}{j}) & \ge(1/6+ \beta /2)n\,, \\
			\deg_{G'_\r}(v,\I{i}{3-j}) & \ge(1/6+ \beta /2)n\,, \\
			\deg_{G'_c}(v,\I{i}{j})    & \ge(1/6+ \beta /2)n\,,
		\end{align*}
		where $c$ is blue if $i=j$, and $c$ is red  otherwise.
	\end{claim}
	\begin{claimproof}
		We will show the third part for a vertex $v \in  \I{1}{1}$.
		The other cases follow by analogous arguments.
		By~\eqref{equ:size-V_0'}, all but $|V_0'|\le 5\sqrt{\gamma} n$ vertices of $G$ are covered by $\bigcup_{i,j \in [2]} \I{i}{j}$.
		Suppose that $v \in V_x$ for some  $x \in \cI{i}{j} \sm \cQ_{1,1}$.
		By definition of $\cQ_{1,1}$, it follows that $|N_{\RdG_\r}(x) \cap \cI{1}{1} | \leq \sqrt{\gamma} n$.
		Moreover, we have $|V_y| \leq n/m$ for each $y \in \cI{1}{1}$.
		This and~\eqref{equ:Iij<1/4} implies that
		\begin{align*}
			\deg_{G'_\b}(v,\I{1}{1}) & = (2/3+\beta)n -|\I{1}{2}| - |\I{2}{1} | - 5\sqrt{\gamma} n - \sqrt{\gamma} m\cdot n/m
			\\ &\geq (1/6 + \beta /2)n\,.\tag*{\qedhere}
		\end{align*}
	\end{claimproof}

	Note that since $(\R_1,\R_2)$ does not admit red or blue bridges, there is at
	most one vertex  $x_\r^*\in V_0'$ with $dn$ red $G''$-neighbours in $\I{1}{1} \cup
		\I{1}{2}$ and $dn$ red $G''$-neighbours in $\I{2}{1} \cup \I{2}{2}$. If there is no such vertex, we say $x_\r^*$ does not exist.
	Similarly, there is at most one vertex $x_\b^*\in V_0'$ with $dn$ blue $G''$-neighbours in $\I{1}{1} \cup \I{2}{1}$ and $dn$ blue $G''$-neighbours in $\I{1}{2} \cup \I{2}{2}$, and if there is no such vertex we say $x_\b^*$ does not exist. Let~$X^*$ be the set containing the existing vertices among $x^*_\r$ and $x^*_\b$ (so $0\le|X^*|\le 2$).

	We shall next distribute the exceptional vertices in~$V_0'$.

	\begin{claim}\label{cla:I'}
		We can partition the vertices of $V_0' \sm X^*$ into sets
		$\{\I{i}{j}'\}_{i,j \in [2]}$ (of size at most $5 \sqrt{\gamma} n$)
		such that every $v \in \I{i}{j}'$ satisfies
		\begin{equation*}
			\deg_{G''_\b}(v,\I{3-i}{j}) \ge({1}/{6} +\beta/2)n \,, \qquad
			\deg_{G''_\r}(v,\I{i}{3-j}) \ge({1}/{6} +\beta/2)n \,.
		\end{equation*}
	\end{claim}
	\begin{claimproof}
		Let {$v\in V_0'\setminus X^*$}. Suppose that $v$ has at least $2dn$ neighbours in each $\I{i}{j}$ with $i,j\in[2]$, so that $v$ sends at least $dn$ edges in one or the other (or possibly both) colours to each $\I{i}{j}$. If there are three or more $\I{i}{j}$ such that $v$ has $dn$ or more red neighbours in each, then by definition we would have $v\in X^*$; the same holds for three or more $\I{i}{j}$ such that $v$ has $dn$ or more red neighbours in each. It follows that there are some two $\I{i}{j}$ to which $v$ sends less than $dn$ red edges, and because~$v$ is not in~$X^*$ these two sets come from a blue component, namely they are~$\I{1}{j}$ and~$\I{2}{j}$ for some~$j\in[2]$. Now~$v$ must send more than~$dn$ red edges to each of the sets~$\I{1}{3-j}$ and~$\I{2}{3-j}$. But these two sets are in different red components, so~$v\in X^*$. This contradiction shows that there are $i,j \in [2]$, such
		that $\deg(v,\I{i}{j}) < 2dn$.
		Moreover, all but {$|V_0'|\leq 5\sqrt{\gamma} n$} vertices are
		covered by $\bigcup_{i,j \in [2]} \I{i}{j}$. By~\eqref{equ:Iij<1/4}, this
		implies that $v$ has at least
		$$(2/3+\beta)n- 2(1/4+{5\sqrt{\gamma})n - 2dn - 5\sqrt{\gamma} n}\geq ( {1}/{6} +\beta  /2)n$$
		neighbours in each of $\I{3-i}{j}$ and $\I{i}{3-j}$. As $v\notin X^*$, we cannot have that $v$ sends $dn$ red edges to both $\I{i}{3-j}$ and $\I{3-i}{j}$, nor at least $dn$ blue edges to both. Thus all but at most $dn$ edges from $v$ to $\I{i}{3-j}$ are of one colour, and all but at most $dn$ edges from $v$ to $\I{3-i}{j}$ are of the other colour, and we can add $v$ either to $\I{i}{j}'$ or
		to $\I{3-i}{3-j}'$.
	\end{claimproof}

	Now set $\J{i}{j}=\I{i}{j}\cup\I{i}{j}'$ for $i,j\in[2]$. Note that by definition the sets $\J{i}{j}$, together with $X^*$, partition the vertices of $G$. Observe that
	by~\eqref{equ:Iij<1/4} and Claim~\ref{cla:I'} we have
	\begin{equation}\label{equ:Jij<1/4}
		(1/4-5\sqrt{\gamma})n\le |\J{i}{j}|=
		|\I{i}{j}|+|\I{i}{j}'|\le(1/4+5\sqrt{\gamma})n+ 5\sqrt{\gamma} n \le (1/4+10\sqrt{\gamma})n\,.
	\end{equation}

	Suppose $|\J{1}{1}|\ge|\J{2}{1}|$. Then either we find a partition of $V(G)$ into three cycles, or $G$ has one of two rather special structures. This claim, together with a short case distinction, is enough for Case 2: the claim (with sets $\J{i}{j}$ and colours changed appropriately) applies equally if $|\J{2}{1}|\ge|\J{2}{2}|$, and so on; and we will argue that in at least one of these four situations we do find the cycle partition. Let us now state precisely what structural information we obtain.
	\begin{claim}\label{cla:nopart-structure}
		Suppose $|\J{1}{1}|\ge|\J{2}{1}|$. Then one of the following three statements holds.
		\begin{enumerate}[label=(\roman*)]
			\item There are three cycles in $G$, two blue and one red, whose vertex sets partition $V(G)$.
			      \item\label{nopart-structure:ii} $x_\r^*$ exists and has less than $\tfrac12dn$ neighbours in $\J{1}{2}$, and $|\J{1}{1}|=|\J{2}{1}|$.
			      \item\label{nopart-structure:iii} $x_\r^*$ exists and has less than $\tfrac12dn$ neighbours in $\J{1}{2}$, and $|\J{1}{1}|=|\J{2}{1}|+1$, and either $x_\b^*$ does not exist or it exists and has less than $\tfrac12dn$ blue neighbours in $\J{2}{1}$.
		\end{enumerate}
	\end{claim}

	Before we prove this claim, we explain why it is enough to complete the proof of Case 2 of Lemma~\ref{lem:extremal-colouring}. To begin with, we may suppose without loss of generality that $\J{1}{1}$ is of maximum size among the $\J{i}{j}$. Then Claim~\ref{cla:nopart-structure} (as written above) applies. Either we are done, or $G$ has one of the structures described in~\ref{nopart-structure:ii} and~\ref{nopart-structure:iii} above. We distinguish these two separate cases. 
	Note that in both situations $x_\r^*$ exists and has less than $\tfrac12dn$ neighbours in $\J{1}{2}$. 
	We will only need to use the extra structure given by~\ref{nopart-structure:iii} at this point; within the following two cases, when we apply Claim~\ref{cla:nopart-structure} we will simply conclude that either we obtain a cycle partition, or the two sets differ in size by at most one and $x_\r^*$ exists and has less than $\tfrac12dn$ neighbours in $\J{1}{2}$.
	
	\medskip
	\textbf{Case 2.1}: Suppose $x_\r^*$ exists and has less than $\tfrac12dn$ neighbours in $\J{1}{2}$, and $|\J{1}{1}|=|\J{2}{1}|$. Because $\J{1}{1}$ is maximum size, we have $|\J{2}{1}|\ge|\J{2}{2}|$. This means we can again apply Claim~\ref{cla:nopart-structure}. Either we obtain the desired cycle partition, or $|\J{2}{1}|=|\J{2}{2}|$, or $|\J{2}{1}|=|\J{2}{2}|+1$. Either way, we have that $x_\b^*$ exists and has less than $\tfrac12dn$ neighbours in $\J{2}{2}$. If $|\J{2}{2}|\ge|\J{1}{2}|$ we can yet again apply Claim~\ref{cla:nopart-structure}, and it necessarily returns a partition into three cycles, since otherwise $x_\r^*$ must have less than $\tfrac12dn$ neighbours in both $\J{1}{2}$ and $\J{2}{1}$, which violates the minimum degree condition of $G$. We have $|\J{2}{2}|\ge|\J{1}{2}|$ except when $|\J{2}{1}|-1=|\J{2}{2}|=|\J{1}{2}|-1$, so it remains to consider this final possibility. In this case we have $|\J{1}{2}|\ge|\J{1}{1}|$ (in fact, with equality). That means we can yet again apply Claim~\ref{cla:nopart-structure}. Again it necessarily returns the desired partition into three cycles since otherwise we are told that $x_\b^*$ has less than $\tfrac12dn$ neighbours in $\J{2}{2}$. Since we already know it has less than $\tfrac12dn$ neighbours in $\J{1}{1}$ this violates the minimum degree condition of $G$.
	
\medskip
	\textbf{Case 2.2}: Suppose that $x_\r^*$ exists and has less than $\tfrac12dn$ neighbours in $\J{1}{2}$, and $|\J{1}{1}|=|\J{2}{1}|+1$, and either $x_\b^*$ does not exist or it does exist and has fewer than $\tfrac12dn$ blue neighbours in~$\J{2}{1}$. If~$x_\b^*$ exists, then it links~$B_1$ and $B_2$ in blue, so it must have at least $\tfrac12dn$ blue neighbours in $\J{1}{1}$. The point is that either $x_\b^*$ does not exist, or it does but has at least~$\tfrac12dn$ neighbours in~$\J{1}{1}$. This means that if $|\J{2}{1}|\ge|\J{2}{2}|$, then Claim~\ref{cla:nopart-structure} necessarily returns the desired three-cycle partition and we are done. Thus we have $|\J{2}{1}|<|\J{2}{2}|$, and this is only possible if $|\J{2}{2}|=|\J{1}{1}|$. So $\J{2}{2}$ is of maximum size, and we have $|\J{2}{2}|\ge|\J{1}{2}|$, which is what we need to apply Claim~\ref{cla:nopart-structure}. As in Case 1, we must obtain a three-cycle partition, since otherwise $x_\r^*$ violates the minimum degree condition on $G$.

	\medskip

	What remains to do is to show that Claim~\ref{cla:nopart-structure} holds. There are a few ways we might obtain the required cycle decomposition; in any case, the first thing we need to do is construct a suitable red cycle, and the following claim provides several possibilities for this.

	\begin{claim}\label{cla:one-red}
		The following three types of red cycle $C_\r$ are in $G$.
		\begin{itemize}
			\item $C_\r$ is contained in $\J{1}{2}$.
			\item if $x_\r^*$ has at least $\tfrac14dn$ red neighbours in $\J{1}{2}$, then $C_\r$ contains $x_\r^*$ and its remaining vertices are in $\J{1}{2}$.
			\item if $x_\r^*$ has at least $\tfrac14dn$ red neighbours in $\J{1}{1}$, then $C_\r$ contains $x_\r^*$, two vertices of $\J{1}{1}$, and its remaining vertices are in $\J{1}{2}$.
		\end{itemize}
		Furthermore, in each case we have $|V(C_\r)|\le 20\sqrt{\gamma} n$, $\big|\J{1}{2}\setminus V(C_\r)\big|\le |\J{2}{2}|-2$. Additionally, for each $c\in\{\r,\b\}$, if $x_c^*$ exists and has more than $100$ blue-neighbours in $\J{1}{2}$, then $V(C_\r)$ covers at most half of the blue-neighbours of $x_c^*$ in $\J{1}{2}$.
	\end{claim}
	\begin{claimproof}
		For the first type of cycle, we pick an edge $uu'$ in $\I{1}{2}$ and let $P$ be the path $(u,u')$. For the second, we choose two distinct red neighbours $u,u'$ in $\I{1}{2}$ of $x_\r^*$ and let $P=(u,x_\r^*,u')$. For the third, we pick two distinct red neighbours $v,v'$ of $x_\r^*$ in $\I{1}{1}$, let $u$ be a red-neighbour of~$v$ in~$\I{1}{2}$, and $u'\neq u$ be a red-neighbour of $v'$ in~$\I{1}{2}$, and set $P=(u,v,x_\r^*,v',u')$. In each case, we have a path $P$ on at most five vertices whose ends are in $\I{1}{2}$.

		We now extend $P$ to the desired cycle $C_\r$ as follows. We first let $S^*$ be a vertex set containing~$V(P)$, together with some other vertices as follows. For each $c\in\{\r,\b\}$, if~$x_c^*$ has between~$100$ and $\tfrac{n}{1000}$ blue-neighbours in $\J{1}{2}$, we add an arbitrary half of them, not contained in $P$, to $S^*$; this is possible since $P$ contains at most five vertices. In total we have $|S^*|\le2\tfrac{n}{2000}+9<\tfrac{n}{100}$. We let $\ell=\max\big(2,2+|\J{1}{2}|-|\J{2}{2}|\big)$. By
		Claim~\ref{cla:I-profiles} we have
		$\deg_{G''_\r}(u,\I{1}{2}),\deg_{G''_\r}(u,\I{1}{2})\ge(1/6+\beta/2)n$.
		Since $|\I{1}{2}|=(1/4\pm 5\sqrt{\gamma})n$ by~\eqref{equ:Iij<1/4}, this implies
		that~$u$ and~$u'$ have a common neighbour~$u_1$ in $\I{1}{2}\setminus( S^*\cup\{u,u'\}$, hence
		we get a path $u,u_1,u'$ of length $\ell=2$. We can then repeat this
		argument to get a common neighbour $u_2$ of~$u_1$ and~$u'$ in
		$\I{1}{2}\setminus\big(S^*\cup\{u\}\big)$ and hence a path $u,u_1,u_2,u'$ of length
		$\ell=3$, and so on. In general we thus obtain a path
		$u,u_1,\dots,u_{\ell-1},u'$ of length~$\ell$ by extending the path of length
		$\ell-1$ by finding a common neighbour~$u_{\ell-1}$ of~$u_{\ell-2}$ and~$u'$
		in $\I{1}{2}\setminus\big(S^*\cup\{u,u_1,\dots,u_{\ell-3}\}\big)$, which we
		certainly can do since by~\eqref{equ:Jij<1/4} the difference in size between $\J{1}{2}$ and $\J{2}{2}$ is at most {$15\sqrt{\gamma} n$, so $\ell\le {17\sqrt{\gamma} n}$}.

		We claim this $C_\r$ has all the required properties. What remains to check is that $|V(C_\r)|\le {20\sqrt{\gamma} n}$, that $\big|\J{1}{2}\setminus V(C_\r)\big|\le|\J{2}{2}|-2$, and that it does not cover too many blue-neighbours of the vertices of $X^*$. The first two of these follow by definition of $\ell$. So we need to check $C_\r$ does not cover too many blue-neighbours of the vertices of $X^*$. If $x_c^*\in X^*$ has less than $100$ blue-neighbours in $\J{1}{2}$, there is nothing to check. If $x_c^*$ has between $100$ and $\tfrac{n}{1000}$ blue-neighbours in $\J{1}{2}$, then at least half are in $S^*\setminus V(P)$ and hence not covered by $C_\r$ as desired. If $x_c^*$ has more than $\tfrac{n}{1000}$ blue-neighbours in $\J{1}{2}$, then at most $|V(C_\r)|\le {20\sqrt{\gamma} n}$ are covered by $C_\r$, which is less than half as desired.
	\end{claimproof}

	Next, we need to construct the two blue cycles covering the vertices left over after $C_\r$ is constructed. We need one cycle contained in $\J{1}{1}\cup\J{2}{1}\cup X^*$ and another in $\J{1}{2} \cup \J{2}{2}\cup X^*$. These two sets satisfy the same conditions, and the following lemma applies to both (though we state it for $\J{1}{1}\cup\J{2}{1}\cup X^*$). Since at this point $C_\r$ is not fixed (Claim~\ref{cla:one-red} above gives several possible constructions) we simply avoid a vertex set $S$ of size at most ${20\sqrt{\gamma} n}$.

	\begin{claim}\label{cla:one-blue}
		Let $S$ be a set of at most ${20\sqrt{\gamma} n}$ vertices. Choose any (possibly empty) disjoint subsets $A,B\subseteq X^*\setminus S$ with the following property: $x_c^*$ may only be in $A$ if it has at least $\tfrac18dn$ blue-neighbours in $\J{2}{1}\setminus S$, and $x_c^*$ may only be in $B$ if it has at least $\tfrac18dn$ blue-neighbours in $\J{1}{1}\setminus S$. Suppose that
		\[\big|A\cup\J{1}{1}\setminus S\big|\ge\big|B\cup\J{2}{1}\setminus S\big|\,.\]
		Then there is a blue cycle in $G$ whose vertex set is $\big(A\cup B\cup\J{1}{1}\cup\J{2}{1}\big)\setminus S$.
	\end{claim}
	\begin{claimproof}
		We begin by constructing disjoint blue paths outside $S$ greedily as follows. Given a vertex~$x$ of $A$, choose two blue-neighbours in $\I{2}{1}\setminus S$, giving a three-vertex blue path containing $x$ with ends in $\I{2}{1}\setminus S$; extend this to a five-vertex blue path with ends in $\I{1}{1}\setminus S$ by choosing blue neighbours in $\I{1}{1}\setminus S$. This is possible by the condition on $A$, since $|\J{i}{j}|\le|\I{i}{j}|+ {5\sqrt{\gamma} n}$, and by Claim~\ref{cla:I-profiles}. Given a vertex $x$ of $B$, we create a three-vertex blue path containing $x$ with ends in $\I{1}{1}\setminus S$ by simply choosing two blue-neighbours of $x$ in $\I{1}{1}\setminus S$. Observe that at each step we have more than~$10$ choices for the next vertex to pick; since $|A\cup B|\le 2$ the paths we construct have in total at most $10$ vertices, so we can perform the construction vertex-disjointly.

		At this point we have $|A\cup B|$ disjoint blue paths with ends in $\I{1}{1}\setminus S$. If this is zero paths, choose a blue edge $uu'$ from $\I{1}{1}\setminus S$ to $\I{2}{1}\setminus S$ and let $P=(u,u')$. If it is one path, let $u$ be one end and let $u'$ be a blue-neighbour in $\I{2}{1}\setminus S$ not used on the path: this gives a blue path~$P$ from $u$ to $u'$. If it is two paths, suppose they go from $u$ to $v$ and $v'$ to $v''$. Choose a common blue-neighbour~$w$ of~$v$ and~$v'$ in~$\I{2}{1}\setminus S$, and a different blue neighbour~$u'$ of~$v''$ in $\I{2}{1}\setminus S$, which are not used on the paths: this gives a blue path $P$ from $u$ to $u'$. Each of these constructions is possible by Claim~\ref{cla:I-profiles}. In any case, $P$ uses the same number of vertices in $\J{1}{1}\cup A$ as in $\J{2}{1}\cup B$.

		We now apply Lemma~\ref{lem:bip-path} to find a blue path from $u$ to $u'$ covering what remains of $\J{1}{1}\setminus S$ and $\J{2}{1}\setminus S$. Specifically, the two sets we need to find a Hamilton path in are $\big(\J{1}{1}\setminus (S\cup V(P))\big)\cup\{u\}$ and $\big(\J{2}{1}\setminus (S\cup V(P))\big)\cup\{u'\}$, with $I'_{1,1}\setminus (S\cup V(P))$ the vertices which do not necessarily have many blue neighbours in $\J{1}{1}$. The required degree conditions hold by Claims~\ref{cla:I-profiles} and~\ref{cla:I'}, and since $S$ is small. By assumption $V(P)$ covers an equal number of vertices in $A\cup \J{1}{1}\setminus S$ and $B\cup\J{2}{1}\setminus S$, so that
		\[\big|\big(\J{1}{1}\setminus S\cup V(P)\big)\cup\{u\}\big|\ge \big|\big(\J{2}{1}\setminus S\cup V(P)\big)\cup\{u'\}\big|\] as required.
	\end{claimproof}

	We are now in a position to prove Claim~\ref{cla:nopart-structure}, which we do by considering a few cases depending on which exist of, and the coloured neighbourhoods of, $x_\r^*$ and $x_\b^*$.
	\begin{claimproof}[Proof of Claim~\ref{cla:nopart-structure}]
		To begin with, suppose that $x_\r^*$ does not exist. We apply Claim~\ref{cla:one-red} to find a red cycle $C_\r$ contained in $\J{1}{2}$ satisfying the properties listed there. We apply Claim~\ref{cla:one-blue}, with $S=A=B=\emptyset$, to find a blue cycle on vertices $\J{1}{1}\cup\J{2}{1}$. We now want to apply Claim~\ref{cla:one-blue} a second time, with $S=V(C_\r)$, to find a cycle on vertices $X^*\cup\J{2}{2}\cup\J{1}{2}\setminus S$. To do this, we assign $x_\b^*$, if it exists, to one of $A$ and $B$ such that the condition of Claim~\ref{cla:one-blue} is satisfied. This is possible for the following reason. If $x_\b^*$ exists, it links the blue components and in particular has at least $dn$ blue neighbours in $\J{1}{2}\cup\J{2}{2}$. So it has at least $\tfrac12dn$ blue-neighbours in one of these sets, and so (by Claim~\ref{cla:one-red}) in must have at least $\tfrac14dn$ blue-neighbours in one of $\J{1}{2}\setminus S$ and~$\J{2}{2}$. These three cycles form the desired partition.

		A very similar argument handles the case that $x_\r^*$ exists and has $\tfrac14dn$ red neighbours in $\J{1}{2}$; under this condition Claim~\ref{cla:one-red} allows us to choose $C_\r$ contained in $\J{1}{2}\cup\{x_\r^*\}$ covering $x_\r^*$, and we then obtain the remaining two blue cycles as above verbatim.

		Again, if $x_\r^*$ exists and has $\tfrac14dn$ blue neighbours in $\J{1}{2}$, we can perform a similar argument to the first. We obtain $C_\r$ contained in $\J{1}{2}$ and a blue cycle on vertices $\J{1}{1}\cup\J{2}{1}$ exactly as there; we obtain a blue cycle on vertices $X^*\cup\J{2}{2}\cup\J{1}{2}\setminus V(C_\r)$ by following the same argument as above, except that we add $x_\r^*$ to $A$.

		If $x_\r^*$ exists and has $\tfrac12dn$ neighbours in $\J{1}{2}$, at least half are red or half are blue and so one of these two cases occur.

		We are left to consider the case that $x_\r^*$ exists but has fewer than $\tfrac2dn$ neighbours in $\J{1}{2}$. Since $x_\r^*$ by definition has at least $dn$ red neighbours in $\J{1}{1}\cup\J{1}{2}$, in this case $x_\r^*$ must have at least $\tfrac12dn$ red neighbours in $\J{1}{1}$.

		Suppose first that $|\J{1}{1}|\ge|\J{2}{1}|+2$. We apply Claim~\ref{cla:one-red} to find a red cycle $C_\r$ whose vertices are in $\J{1}{2}$, apart from two in $\J{1}{1}$ and $x_\r^*$, satisfying the properties listed there. We apply Claim~\ref{cla:one-blue}, with $S=V(C_\r)$ and $A=B=\emptyset$, to find a blue cycle on vertices $\J{1}{1}\cup\J{2}{1}\setminus V(C_\r)$. We now want to apply Claim~\ref{cla:one-blue} a second time, with $S=V(C_\r)$, to find a cycle on vertices $(X^*\setminus\{x_\r^*\})\cup\J{2}{2}\cup\J{1}{2}\setminus S$. To do this, we assign $x_\b^*$, if it exists, to one of $A$ and $B$ such that the condition of Claim~\ref{cla:one-blue} is satisfied; this is possible for the same reason as in the first case above. These three cycles form the desired partition.

		If $|\J{1}{1}|=|\J{2}{1}|$, then we have the structure described in~\ref{nopart-structure:ii}.

		Finally, suppose $|\J{1}{1}|=|\J{2}{1}|+1$. If $x_\b^*$ does not exist, or it exists and has less than $\tfrac12dn$ blue neighbours in $\J{2}{1}$, we have the structure of~\ref{nopart-structure:iii}. So we only need to show that if $x_\b^*$ exists and has at least $\tfrac12dn$ blue neighbours then we obtain the desired three-cycle partition. We do this as follows. We first use Claim~\ref{cla:one-red} to find a red cycle $C_\r$ which contains $x_\r^*$, two vertices of $\J{1}{1}$, and vertices of $\J{1}{2}$ such that $|\J{1}{2}\setminus V(C_\r)|\le|\J{2}{2}|-2$. We now apply Claim~\ref{cla:one-blue}, with $A=B=\emptyset$, to find a blue cycle covering $(\J{1}{2}\cup\J{2}{2})\setminus V(C_\r)$. We apply Claim~\ref{cla:one-blue} again, with $A=\{x_\b^*\}$ and $B=\emptyset$, to find a blue cycle covering $(A\cup\J{1}{1}\cup\J{2}{1})\setminus V(C_\r)$; this second application is possible since $|A\cup \J{1}{1}\setminus V(C_\r)|=|\J{1}{1}|-2+1=|\J{2}{1}|$. This gives the desired partition into three cycles.
	\end{claimproof}
	As explained above, Claim~\ref{cla:nopart-structure} completes the proof.
\end{proof}


\section{Concluding Remarks}\label{sec:remarks}

There are several natural questions arising from our result.
First of all, we believe that the error term $\beta$ in Theorem~\ref{thm:main} can be removed, as conjectured by Pokrovskiy~\cite{Pok16}.
However, this will likely require a careful and technically difficult analysis of the underlying extremal cases.

One might also ask how many monochromatic cycles are necessary to partition a graph with minimum degree below $2n/3$.
Kor\'andi, Lang, Letzter and Pokrovskiy~\cite{KLLP20} showed that $r$-edge-coloured graphs of minimum degree $n/2 + \Omega(\log n)$ can be covered with at most $O(r^2)$ disjoint monochromatic cycles.
This is essentially tight in terms of the degree condition, as with a minimum degree bound below  $n/2 + O(\log n / \log \log n)$ there are colourings where the number of cycles can no longer be bounded in terms of $r$.
On the other hand, it is not known up to which minimum degree one can always partition into four monochromatic cycles.
\begin{problem}
Determine the smallest $\delta > 0$ for which every $2$-edge-coloured graph $G$ on $n$ vertices and $\delta(G) \geq \delta n$ can be partitioned into four monochromatic cycles.
\end{problem}

An old theorem of Ore~\cite{Ore60} states that every graph on $n$ vertices contains a Hamilton cycle, provided the graph satisfies $\deg(x) + \deg(y) \geq n$ for every two non-adjacent vertices~$x$ and $y$.
Bar\'{a}t and S\'{a}rk\"{o}zy~\cite{BS16} asked whether something similar is true for monochromatic cycle partitions.
They conjectured that any red and blue edge-coloured graph $G$  on $n$ vertices can be partitioned into a red and a blue monochromatic cycle, provided that  $\deg(x) + \deg(y) > 3n/2$ for every two non-adjacent vertices $x$ and $y$.
In support of their conjecture, they proved that under the stronger assumption $\deg(x) + \deg(y) > (3/2 + o(1))n$ almost all vertices can be covered.
We believe that our techniques are suitable to provide a covering of the remaining vertices.
\begin{problem}
For $\beta >0$ and sufficiently large $n$, prove that any red and blue edge-coloured graph $G$ on $n$ vertices can be partitioned into a red and a blue monochromatic cycle, provided that  $\deg(x) + \deg(y) > (3/2 + \beta)n$ for every two non-adjacent vertices $x$ and $y$.
\end{problem}

Very recently, Arras~\cite{Arras} considered an approximate analogue of this problem for
three cycles: He proved that if $\deg(x) + \deg(y)>4/3n+o(n)$ for each pair of
non-adjacent vertices $x$ and $y$, then all but $o(n)$ vertices can be
partitioned into three monochromatic cycles.
Note that this is an (approximate) generalization of our main result.

It would also be interesting to determine what happens if we have more than 2 colours. That is, we would like to know which minimum degree conditions guarantee a partition of an $r$-edge-coloured graph into $f(r)$ monochromatic cycles, where $f(r)$ is some function.
In this direction, Kor\'andi, Lang, Letzter and Pokrovskiy~\cite{KLLP20} showed for any fixed $\delta>1/2$ and growing $r$, there are families of $n$-vertex graphs of minimum degree $\delta n$ that cannot be partitioned into fewer than~$\Omega(r^2)$ monochromatic  cycles.
(In contrast to this, the complete graph can be partitioned into $O(r \log r)$ cycles~\cite{GRSS06}.)
However, it is not clear what happens for smaller $r$.
For instance, for $r=3$  it might be possible to show an analogue of our main theorem with one more colour and one more cycle.
\begin{conjecture}\label{conj:Ore}
	Any $3$-edge-coloured graph on $n$ vertices and of minimum degree at least $3n/4$ can be partitioned into $4$ cycles.
\end{conjecture}

This would be tight, as the following easy example shows. For $m\in \mathbb N$, consider the graph on sets $A$, $B$, $C$, $D$, with $|A|=|C|=m+2$, $|B|=m$, $|D|=m+1$, having all edges except the ones inside $B$, inside $D$, and between $A$ and $C$. Colour all edges inside $A$, inside $C$, and between $B$ and $D$ red, colour all edges between $A$ and $B$, and between $C$ and $D$ green, and colour all edges between $A$ and $D$, and between $B$ and $C$ blue. This graph has $n=4m+5$ vertices, minimum degree  $3m+2=\lfloor3n/4\rfloor-1$, and cannot be partitioned into  less than four monochromatic cycles.

P\'osa proved that a graph $G$ with degree sequence $d_1 \leq \dots \leq d_n$ and $d_i > i$ for every $i < n/2$ contains Hamilton cycle.
We believe that the following analogue of this result for monochromatic cycle partitions is true. 
Very recently, Arras~\cite{Arras} obtained an approximate solution for this in the case $r=2$.

\begin{conjecture}
	There is a function $f$ with the following property.
	For $\beta > 0$, $r \geq2$ and $n$ sufficiently large, let  $G$ be a graph with degree sequence $d_1 \leq \dots \leq d_n$ such that $d_i \geq i + \beta n$ for every $i\leq n/2$.
	Then every $r$-edge-colouring of $G$ admits a partition into at most $f(r)$ monochromatic cycles.
\end{conjecture}

Finally, recent contributions in extremal combinatorics have been concerned with finding sufficient conditions for spanning substructures that are in some sense `closer' to being necessary than assumptions on the minimum degree~\cite{CKL+16,HJL22,KM15,LS21}.
It would be very interesting to determine whether similarly abstract properties can be identified for monochromatic cycle partitioning, although we do not have a concrete suggestions of how such properties may look like.


\bibliographystyle{amsplain}
\bibliography{monochromatic-cover}

\end{document}